\documentclass[11pt,a4paper,reqno]{amsart}  

\usepackage[utf8]{inputenc}
\usepackage[english]{babel}
\usepackage[margin=3cm]{geometry}

\usepackage[usenames,dvipsnames,svgnames,table]{xcolor}
\usepackage{float}
\usepackage{caption,subcaption}
\usepackage{tikz,pgf}

\usepackage{fancyhdr}
\usepackage{footmisc}
\usepackage{lipsum}

\usepackage{blindtext}

\usepackage{amsfonts,amsmath,amssymb,amsthm,mathtools,mathrsfs}
\usepackage{array,caption,subcaption,graphicx}
\usepackage[shortlabels]{enumitem}

\usepackage{hyperref}
\usepackage{xcolor}

\definecolor{unbleu}{rgb}{0.03, 0.15, 0.4}

\hypersetup{
pdfborder = {0 0 0},
colorlinks,
linkcolor=unbleu,
citecolor=unbleu,
urlcolor=unbleu
}

\usepackage{tikz-cd}

\theoremstyle{plain}
\newtheorem{theorem}{Theorem}[section]
\newtheorem*{theorem*}{Theorem}
\newtheorem{lemma}[theorem]{Lemma}
\newtheorem*{lemma*}{Lemma}

\newtheorem*{proposition*}{Proposition}
\newtheorem{corollary}[theorem]{Corollary}
\newtheorem*{corollary*}{Corollary}

\theoremstyle{definition}
\newtheorem{definition}[theorem]{Definition}
\newtheorem{example}[theorem]{Example}
\newtheorem{remark}[theorem]{Remark}

\newcommand{\func}[5]{#1\colon \begin{array}[t]{ >{\displaystyle}r >{{}}c<{{}}  >{\displaystyle}l } #2 &\to& #3 \\ #4 &\mapsto& #5 \end{array}}

\newcommand{\norm}[1]{\lVert #1 \rVert}
\newcommand{\modl}[1]{\lvert #1 \rvert}
\newcommand\given[1][]{\:#1\vert\:}

\DeclarePairedDelimiter{\ceil}{\lceil}{\rceil}
\DeclarePairedDelimiter{\floor}{\lfloor}{\rfloor}

\setlist[enumerate]{topsep=0pt}
\setlist[itemize]{topsep=0pt}
\numberwithin{equation}{section}

\usepackage[T2A,T1]{fontenc}
\DeclareSymbolFont{cyrillic}{T2A}{cmr}{m}{n}
\DeclareMathSymbol{\D}{\mathalpha}{cyrillic}{196}

\makeatletter
\@namedef{subjclassname@2020}{\textup{2020} Mathematics Subject Classification}
\makeatother

\begin{document}

\title{Functional Limit Theorems for Dynamical Systems with Correlated Maximal Sets}

\author[R. Couto]{Raquel Couto}
\address{Raquel Couto\\ Centro de Matem\'{a}tica \& Faculdade de Ci\^encias da Universidade do Porto, Rua do Campo Alegre 687, 4169-007 Porto, Portugal}
\email{\href{mailto:up201106786@fc.up.pt}{up201106786@fc.up.pt}}
\urladdr{\url{https://rbscouto.github.io}}

\thanks{The author was supported by an FCT -- Funda\c{c}\~ao para a Ci\^encia e a Tecnologia scholarship with reference number PD/BD/150456/2019. The author was also partially supported by FCT projects PTDC/MAT-PUR/4048/2021 and 2022.07167.PTDC, with national funds, and by CMUP, which is financed by national funds through FCT, under the project with reference UIDB/00144/2020. The author acknowledges the institutions Faculdade de Ciências da Universidade do Porto and University of St Andrews where the work was carried out as part of her PhD studies. The author thanks her PhD supervisors Ana Cristina Freitas, Jorge Freitas and Mike Todd for the insightful discussions, careful reading and helpful suggestions.}

\date{\today}

\begin{abstract}
In order to obtain functional limit theorems for heavy tailed stationary processes arising from dynamical systems, one needs to understand the clustering patterns of the tail observations of the process. These patterns are well described by means of a structure called the pilling process introduced recently in the context of dynamical systems. So far, the pilling process has been computed only for observable functions maximised at a single repelling fixed point. Here, we study richer clustering behaviours by considering correlated maximal sets, in the sense that the observable is maximised in multiple points belonging to the same orbit, and we work out explicit expressions for the pilling process when the dynamics is piecewise linear and expanding ($1$-dimensional and $2$-dimensional).
\end{abstract}

\maketitle

\tableofcontents

\section{Introduction}\label{sec:intro}
Functional limit theorems for dynamical systems are statements about the convergence of
\begin{equation}\label{eq:func-lim}
S_n(t):=\sum_{i=0}^{\floor{nt}-1} \dfrac{1}{a_n}\mathbf{X}_i-tc_n, \ t \in [0,1],
\end{equation}
in a suitable space of functions, when the process $(\mathbf{X}_n)_{n \in \mathbb{N}_0}$ is dynamically defined by evaluating an observable function along the orbits of the dynamical system ($(a_n)_{n \in \mathbb{N}}$ and $(c_n)_{n \in \mathbb{N}}$ are appropriate scaling sequences). In \cite{tk10}, necessary and sufficient conditions are given for the convergence of (\ref{eq:func-lim}) to an $\alpha$-stable Lévy process, in the space of càdlàg functions on $[0,1]$, $D([0,1])$, equipped with Skorohod's $J_1$ topology (see also references in \cite{tk10} for other examples of work on the subject).

The theme of \cite[Theorem 2.4]{20enriched} is an enriched functional limit theorem for heavy tailed dynamical sums. Our aim is to illustrate its application for a class of examples previously studied in \cite{16correlated} and \cite{17correlated}.

Let $(\mathcal{X},\mathcal{B}_{\mathcal{X}},\mu,f)$ be a probability preserving (discrete) dynamical system which stands for $(\mathcal{X},\mathcal{B}_{\mathcal{X}},\mu)$ a probability space and $f:\mathcal{X} \to \mathcal{X}$ a transformation that preserves $\mu$. We take $\mathcal{X}$ a $d$-dimensional compact Riemannian manifold, with $\text{dist}(\cdot,\cdot)$ the distance induced by the Riemannian metric, then $\mathcal{B}_{\mathcal{X}}$ is the corresponding Borel $\sigma$-algebra, and $\mu$ is a probability measure that we further assume to be absolutely continuous with respect to Lebesgue measure (\textit{i.e.} \textit{acip}). Then, for an observable $\Psi:\mathcal{X} \to \mathbb{R}^d$, $\mathbf{X}_n:=\Psi \circ f^n$, $n \in \mathbb{N}_0$, defines a stationary stochastic process.

A heavy tailed random variable is known as one with an infinite second moment. In our setting, that is written $\mathbb{E}_{\mu}[\norm{\mathbf{X}_0}^2]=\infty$. Since a value which is far away from the mean dominates a Birkhoff sum, then the study of sums of heavy tailed random variables is tied to the study of maxima. Therefore, the study of sums such as in (\ref{eq:func-lim}), when $(\mathbf{X}_n)_{n \in \mathbb{N}_0}$ is heavy tailed, is suited for the well developed framework of Extreme Value Theory (for dynamical systems).

In the context of extremes for dynamical systems (as in, for example, \cite{10hit,12index,torus15,16correlated} and \cite{hyp21}) $\Psi$ is chosen such that there exists a maximal set, $\mathcal{M}$, for $\norm{\Psi(\cdot)}$ whose high values correspond to entries in small neighbourhoods of $\mathcal{M}$. It is only demanded that $\mathcal{M}$ has zero measure, so $\mathcal{M}$ consisting of a finite set of points, a countable set of points, a submanifold or a fractal set are all permitted. Here, for $\mathcal{I}$ a finite or countable set of indices, let $\Psi$ be defined as
\begin{equation}\label{eq:psi}
\Psi(x):=\displaystyle\sum_{i \in \mathcal{I}} h_i(\text{dist}(x,\xi_i))\dfrac{\Phi^{-1}_{\xi_i}(x)}{\norm{\Phi^{-1}_{\xi_i}(x)}}\mathbf{1}_{W_i}(x)
\end{equation}
in the union of some neighbourhoods of each of the $\xi_i$, $i \in \mathcal{I}$, and equal to zero outside of it, where\begin{itemize}
\item $h_i:[0,+\infty) \to \mathbb{R} \cup \{+\infty\}$ satisfies\begin{enumerate}[(i)]
\item $0$ is a global maximum ($h_i(0)=+\infty$ is allowed);
\item $h_i$ is a strictly decreasing bijection in a neighbourhood of $0$;
\item $h_i$ is of type $g_2$ (see, for instance, section 2.1 of \cite{20enriched}) but, more specifically, there exists $\alpha \in (0,2)$ such that, for all $i \in \mathcal{I}$,
\begin{equation}\label{eq:hi}
h_i(x)=c_ix^{-\frac{1}{\alpha}}
\end{equation}
where $c_i>0$.
\end{enumerate}
\item $\Phi_{\xi_i}:V_i \to W_i$ denotes a diffeomorphism defined on an open ball, $V_i$, around $\xi_i$ in $T_{\xi_i}\mathcal{X}$ (the tangent space at $\xi_i$) onto a neighbourhood, $W_i$, of $\xi_i$ in $\mathcal{X}$ such that $\Phi_{\xi_i}(E^{s,u} \cap V_i)=W_{\xi_i}^{s,u} \cap W_i$.
\end{itemize}
The presence of $\dfrac{\Phi^{-1}_{\xi_i}(x)}{\norm{\Phi^{-1}_{\xi_i}(x)}}$ in the definition of $\Psi$ allows for labelling the observation $\Psi(x)$ with a direction, corresponding to the projection of $x$ in $T_{\xi_i}\mathcal{X}$. In particular, when $\mathcal{X}$ is a subset of $\mathbb{R}$ then $\Psi$ reduces to $\psi(x):=\displaystyle\sum_{i \in \mathcal{I}} h_i(\modl{x-\xi_i})\mathbf{1}_{V_i}(x)$ where $V_i$ is an open interval around $\xi_i$.

We remark that, in general, $h_i$ is assumed to have one of three types of behaviour ($g_1$, $g_2$ or $g_3$), which correspond to the three classical EVL (see \cite{10hit}), but since here we are restricting to the context of \cite[Theorem 2.4]{20enriched} only $g_2$ needs to be considered. In fact, we are even restricting within type $g_2$ as we are taking $h_i$ as in (\ref{eq:hi}).

Point processes have proven to be a very useful tool in the study of extremes, originally recording the times of extreme observations in a prescribed time frame (\cite{hv09,10hit,13compound}) and later both times and magnitudes of observations (\cite{20conv}). In both cases, it was possible to deal with clustering of extremes (\textit{i.e.} the appearance of several extreme observations close together in the time line) which has been related to recurrence properties of the underlying dynamical system, such as periodicity (\cite{12index}). We note that clustering can be seen as a short range dependence of the sequence $(\mathbf{X}_n)_{n \in \mathbb{N}_0}$.

In fact, in the presence of clustering there is no limit to (\ref{eq:func-lim}) in Skorohod's $J_1$ topology. We loosely explain why (see \cite{20enriched} for more details). Several extreme observations showing up close together in time lead up to several jumps of the partial sum process (\textit{i.e.} $S_n(t)$) in a small time interval. Then, due to time compression, the latter end up as a sequence of points on the same time instant which is not $J_1$-equivalent to any element of $D([0,1])$. Opting for a weaker topology, such as Skorohod's $M_1$ or $M_2$ topologies, although possibly allowing for limits, implies the loss of the clustering patterns. In other words, once the limit is established, it is impossible to retrieve the data corresponding to the intermediate jumps that took place inside the cluster. For a discussion of $J_1$, $M_1$ and $M_2$ topologies see \cite{whitt}.

The distinctive feature of \cite[Theorem 2.4]{20enriched} is allowing for the functional convergence to hold without loss of the clustering patterns. A key role is played by the functional space $F'$ defined in \cite{20enriched} related to Whitt's space $F$ (\cite{whitt}). Very briefly, $F'$ is made of excursion triples $(V,S^V,(e_V^{s})_{s \in  S^V})$, where $V \in D([0,1])$, $S^V$ is an at most countable set containing the discontinuities of $V$, and $e_V^s$ is the excursion at $s \in S^V$, which lives in a quotient space of $D([0,1])$, and enjoys the essential property that $e_V^s(0)=V(s^-)$ and $e_V^s(1)=V(s)$, meaning that the information regarding the cluster associated to a discontinuity $s$ of the càdlàg function $V$ is recorded in $e_V^s$. Indeed, \cite[Theorem 2.4]{20enriched} is a statement about convergence in $F'$ of heavy tailed dynamical sums, where (\ref{eq:func-lim}) converges to $V$ an $\alpha$-stable Lévy process.

\cite[Theorem 2.4]{20enriched} essentially follows from weak convergence of the cluster point process. We highlight that the cluster point process in \cite{20enriched} differs from previous versions (such as the one presented in \cite{20conv}) by characterising clustering via a piling process: a bi-infinite sequence storing the comparisons of the magnitudes of exceedances in a cluster to the magnitude of an exceedance registered at time $0$ and conditional on the occurrence of such exceedance. Once the piling process is well defined and usual dependence requirements are met by the stochastic sequence, the cluster point process converges to a Poisson point process. Then, if the stochastic sequence has an $\alpha$-regularly varying tail, the functional limit in $F'$ is $(V,S^V,(e_V^{s})_{s \in S^V})$ where $V$ is an $\alpha$-stable Lévy process (whose Lévy-measure is associated to the intensity measure of the Poisson point process - see section 2.5 of \cite{20enriched}). The piling process plays a crucial role in \cite[Theorem 2.4]{20enriched} allowing not only for the characterisation of the exceedances in a given cluster (\textit{i.e.} handling the clustering component on the cluster point process) but also the excursions $(e_V^{s})_{s \in S^V}$ (\textit{i.e.} recording the clustering pattern aside of the functional limit).

We remark that the employment of \cite[Theorem 2.4]{20enriched} when $\mathcal{M}=\{\zeta\}$, where $\zeta$ is a periodic point, is included in \cite{20enriched}. In the context of \cite{16correlated} and \cite{17correlated}, clustering of extreme observations is a consequence of the dynamical link between the elements of the maximal set, rather than being created by periodicity. Specifically, in what we further may refer to as `the setting of multiple correlated maxima', we will be dealing with $\mathcal{M}$ (with more than one element) a finite or countable set whose elements are all in the same orbit, so, alternatively,\begin{enumerate}[(A)]
\item\label{itm:A} $\mathcal{M}=\{\xi_1,\dots,\xi_N\}$ such that there exist $m_1,\dots,m_N$ with $\xi_i=f^{m_i}(\zeta)$, where $\zeta \in \mathcal{X}$, which is compatible with $\mathcal{I}=\{1,\dots,N\}$ ($\mathcal{I}$ as in the definition of $\Psi$ by (\ref{eq:psi}));
\item\label{itm:B} $\mathcal{M}=\{\xi_1,\xi_2,\dots\}=\{\xi_i\}_{i \in \mathbb{N}}$ such that there exist $m_i$, $i \in \mathbb{N}$, with $\xi_i=f^{m_i}(\zeta)$, where $\zeta \in \mathcal{X}$, and $\xi_0=\lim_{i \to \infty} \xi_i$, which is compatible with $\mathcal{I}=\{i: i \in \mathbb{N}\}$;
\end{enumerate}
so that (A) and (B) correspond to the settings of \cite{16correlated} and \cite{17correlated}, respectively. In either case, we further assume that $m_1=0$ (\textit{i.e.} $\xi_1=\zeta$) and require the following for $f$ and $\mu$:\begin{enumerate}[(R1)]
\item\label{itm:R1} along the orbit of $\zeta$, $f$ is $C^1$ and locally invertible;
\item\label{itm:R3} $\mu$ is an \textit{acip} with density $\rho$ satisfying
\begin{equation*}
\lim_{\varepsilon \to 0}\dfrac{\mu(B_{\varepsilon}(x))}{\text{Leb}(B_{\varepsilon}(x))}=\rho(x)
\end{equation*}
which is finite on $\mathcal{M}$; for all $\xi_i  \in \mathcal{M}$ let $D_i \equiv \rho(\xi_i)$.
\end{enumerate}

We are finally able to give a sample of the results in this paper.

The following result is a particular case of Theorem \ref{thm:piling-corr-nper} and appears later as Example \ref{ex:nper-evendist}. It is a case of a finite maximal set as in \ref{itm:A} with $\zeta$ non-periodic. We give analogous results in case \ref{itm:B}.
\begin{theorem}\label{prop:intro}
Let $f(x)=2x \mod 1$, $x \in [0,1]$, and $\mu=Leb$ (invariant for $f$). Take $\zeta=\dfrac{\sqrt{2}}{16}$ (non-periodic), and define the observable $\psi$ as
\begin{equation*}
\psi(x):=\begin{cases}
\modl{x-\zeta}^{-2}, \ x \in B_{\varepsilon_1}(\zeta)\\
\modl{x-f(\zeta)}^{-2}, \ x \in B_{\varepsilon_2}(f(\zeta))\\
\modl{x-f^3(\zeta)}^{-2}, \ x \in B_{\varepsilon_3}(f^3(\zeta))\\
0, \ \text{otherwise}
\end{cases}
\end{equation*}
for some $\varepsilon_1,\varepsilon_2,\varepsilon_3>0$. Observe that, presented as in (\ref{eq:psi}),
\begin{equation*}
\psi(x)=\displaystyle\sum_{i=1}^{3} h_i(\modl{x-\xi_i})\mathbf{1}_{B_{\varepsilon_i}(\xi_i)}(x)
\end{equation*}
for $\xi_1=\zeta$, $\xi_2=f(\zeta)$, $\xi_3=f^3(\zeta)$, and $h_i(t)=t^{-2}$ for $i=1,2,3$ (so that $\alpha=1/2$). In particular, $\mathcal{M}=\{\zeta,f(\zeta),f^3(\zeta)\}$. Then, the piling process is any of the bi-infinite sequences
$$\left(\dots,\infty,U,U.2,\infty,U.2^3,\infty,\dots\right)$$
$$\left(\dots,\infty,U,\infty,U.2^2,\infty,\dots\right)$$
$$\left(\dots,\infty,U,\infty,\dots\right)$$
with probability $\dfrac{1}{3}$, where the entry that is equal to $U$ corresponds to index $0$ and the entries that are visibly different from $\infty$ are the only such entries, for $U$ a random variable uniformly distributed on $[0,1]$.
\end{theorem}
The enriched functional limit theorem follows from the application of \cite[Theorem 2.4]{20enriched} for a piling process given by Theorem \ref{prop:intro}.
\begin{theorem}\label{thm:intro}
Under the hypotheses of Theorem \ref{prop:intro},
\begin{equation*}
S_n(t):=\dfrac{1}{36n^2}\sum_{i=0}^{\floor{nt}-1} X_i, \ t \in [0,1]
\end{equation*}
converges in $F'$ to $(V,\text{disc}(V),(e_V^s)_{s \in \text{disc}(V)})$, where $V$ is an $\alpha$-stable Lévy process ($\alpha=1/2$)
\begin{equation*}
V(t)=\sum_{T_i \leq t}\sum_{j \in \mathbb{Z}} U_i^{-\frac{1}{\alpha}}\mathcal{Q}_{i,j}, \ t \in [0,1],
\end{equation*}
where $T_i$ is uniformly distributed on $\mathbb{R}_0^{+}$, $U_i$ is uniformly distributed on $[0,1]$ and $\mathcal{Q}_i=\xi(\mathbf{\tilde{Q}}_i)$ consists of the equally likely sequences
$$\left(\dots,0,1,1/4,0,1/4^3,0,\dots\right)$$
$$\left(\dots,0,1,0,1/4^2,0,\dots\right)$$
$$\left(\dots,0,1,0,\dots\right)$$
and the excursions are given by $e_V^{T_i}(t)=V(T_i^{-})+U_i^{-2}\displaystyle\sum_{0 \leq j \leq \floor{\tan(\pi(t-\frac{1}{2}))}} \mathcal{Q}_{i,j}$.
\end{theorem}

We structure the document as follows: in Section \ref{sec:corr-setting} we summarise the general theory from \cite{20enriched} that we wish to apply to the setting of multiple correlated maxima and that ultimately leads to the enriched functional limits; in Sections \ref{sec:corr-finite} and \ref{sec:corr-count} we present general statements for the piling processes in the setting of multiple correlated maxima and fully work out some representative examples; finally, in Appendix \ref{appx:dependance} (and \ref{appx:dependance-ex}) we include the dependence requirements.

\section{Background theory}\label{sec:corr-setting}
In this section we summarise the background theory from \cite{20enriched} which we then apply to the setting of multiple correlated maxima in Sections \ref{sec:corr-finite} and \ref{sec:corr-count} to follow. Our main purpose is the application of \cite[Theorem 2.4]{20enriched} (Theorem \ref{thm:func-lim-FFT20} below) to the scenarios given by \ref{itm:A} and \ref{itm:B} in Section \ref{sec:intro}. \cite[Theorem 2.4]{20enriched} follows from the complete convergence of the cluster point process, defined in Section 3.4 of \cite{20enriched} and recalled in Section \ref{subsec:cluster-pp} below, to a Poisson point process (Theorem \ref{thm:pp-conv-FFT20} below - for more details see Sections 3.4 and 4.1 of \cite{20enriched}). The piling process, introduced in Section 3.3 of \cite{20enriched} and recalled in Section \ref{subsec:piling} below, plays a crucial role in describing the clustering patterns and it characterises the limiting Poisson point process (to the cluster point process) as well as the enriched functional limit given by \cite[Theorem 2.4]{20enriched}.

We will always be dealing with a dynamically defined (stationary) process $\mathbf{X}_0,\mathbf{X}_1,\dots$ where $\mathbf{X}_n:=\Psi \circ f^n$, $n \in \mathbb{N}_0$, for a probability preserving $f:\mathcal{X} \to \mathcal{X}$ and a vector valued observable $\Psi$ as described in Section \ref{sec:intro}.

\subsection{Threshold functions}\label{subsec:thre}
Let $(u_n)_{n \in \mathbb{N}}:\mathbb{R}^{+} \to \mathbb{R}^{+}$, where $\mathbb{R}^{+}=(0,+\infty)$, be such that:\begin{enumerate}[(1)]
\item for each $n$, $u_n$ is non-increasing, left continuous and such that
\begin{equation*}
\lim_{\tau_1 \to 0,\tau_2 \to \infty} \mathbb{P}(u_n(\tau_2)<\norm{\mathbf{X}_0}<u_n(\tau_1))=1;
\end{equation*}
\item for each $\tau \in \mathbb{R}^{+}$,
\begin{equation}\label{eq:u_n-tau}
\lim_{n \to \infty} n\mathbb{P}(\norm{\mathbf{X}_0}>u_n(\tau))=\tau.
\end{equation}
\end{enumerate}
$(u_n)_{n \in \mathbb{N}}$ is a normalising sequence of threshold functions with the crucial property expressed in (\ref{eq:u_n-tau}) that the asymptotic frequency of exceedances of a threshold (that depends on $\tau$) is constant (and equal to $\tau$). The use of such normalising sequences is standard even in the classical probabilistic setting of Extreme Value Theory.

Taking generalised inverses of the functions $u_n$, we are able to recover an asymptotic frequency: for every $z \in \mathbb{R}^{+}$, define
\begin{equation}\label{eq:gen-u_n^{-1}}
u_n^{-1}(z)=\sup\{\tau>0: z \leq u_n(\tau)\}.
\end{equation}
Observe that (\ref{eq:gen-u_n^{-1}}) gives $z>u_n(\tau) \implies u_n^{-1}(z) \leq \tau$.

\subsection{Extremal Index}\label{subsec:EI}
The extremal index is a parameter $\vartheta \in [0,1]$ which quantifies the intensity of clustering of extreme observations, that is how close in the time line arbitrarily high observations appear. $\vartheta$ equal to $1$ means absence of clustering of extreme observations which becomes more intense as $\vartheta$ gets closer to $0$.

For a sequence $(q_n)_{n \in \mathbb{N}}$ of positive integers (we discuss the importance of this sequence in Appendix \ref{appx:dependance}), let
\begin{equation}\label{eq:U_n}
U_n(\tau):=\{\norm{X_0}>u_n(\tau)\}
\end{equation}
and
\begin{equation}\label{eq:U_n-q_n}
U_n^{(q_n)}(\tau):=\{\norm{X_0}>u_n(\tau),\norm{X_1} \leq u_n(\tau),\dots,\norm{X_{q_n}} \leq u_n(\tau)\}
\end{equation}
where $(u_n)_{n \in \mathbb{N}}$ is as defined in Section \ref{subsec:thre}.
Then, the extremal index, $\vartheta$, is defined as
\begin{equation}\label{eq:EI}
\vartheta:=\lim_{n \to \infty} \dfrac{\mu(U_n^{(q_n)}(\tau))}{\mu(U_n(\tau))}
\end{equation}
provided the limit exists.

\subsection{Piling process}\label{subsec:piling}
For $(\mathbf{X}_n)_{n \in \mathbb{N}_0}$, for all $s<t$, let $\mathbb{X}_n^{s,t}$ denote the vector
\begin{equation*}
\left(u_n^{-1}(\norm{\mathbf{X}_{s}})\dfrac{\mathbf{X}_{s}}{\norm{\mathbf{X}_{s}}},\dots,u_n^{-1}(\norm{\mathbf{X}_{t}})\dfrac{\mathbf{X}_{t}}{\norm{\mathbf{X}_{t}}}\right)
\end{equation*}
so that, for all $j=s,\dots,t$, the unit vector in the direction of $\mathbf{X}_j$ is scaled by the asymptotic frequency $u_n^{-1}(\norm{\mathbf{X}_j})$.

The piling process is defined on the space
\begin{equation*}
l_{\infty}=\{\mathbf{x}=(x_j)_{j \in \mathbb{Z}} \in (\overline{\mathbb{R}^d}\setminus\{0\})^{\mathbb{Z}}: \lim_{\modl{j} \to \infty} \norm{x_j}=\infty\},
\end{equation*}
being, therefore, a bi-infinite sequence.
\begin{definition}\label{def:piling}
Given a process $(Y_j)_{j \in \mathbb{Z}}$ such that:\begin{enumerate}[(1)]
\item $\mathcal{L}\left(\dfrac{1}{\tau} \mathbb{X}_n^{r_n+s,r_n+t} \given[\Big] \norm{\mathbf{X}_{r_n}} > u_n(\tau)\right) \xrightarrow[n \to \infty]{} \mathcal{L}((Y_j)_{j=s,\dots,t})$, for all $s<t \in \mathbb{Z}$ and all $\tau>0$;
\item the process $(\Theta_j)_{j \in \mathbb{Z}}$ given by $\Theta_j=\dfrac{Y_j}{\norm{Y_0}}$ is independent of $\norm{Y_0}$;
\item $\lim_{\modl{j} \to \infty} \norm{Y_j}=\infty$ a.s.;
\item $\mathbb{P}\left(\inf_{j \leq -1} \norm{Y_j} \geq 1\right)>0$;
\end{enumerate}
the process $(Z_j)_{j \in \mathbb{Z}}$ defined by
\begin{equation*}
\mathcal{L}((Z_j)_{j \in \mathbb{Z}})=\mathcal{L}\left((Y_j)_{j \in \mathbb{Z}} \given[\Big] \inf_{j \leq -1} \norm{Y_j} \geq 1\right)
\end{equation*}
is called the \textit{piling process}.
\end{definition}
In words, (1) states that $(Y_j)_{j \in \mathbb{Z}}$ must be such that, for all $s<t \in \mathbb{Z}$ and all $\tau>0$, $(Y_s,\dots,Y_t)$ has asymptotically the same distribution as $\left(\dfrac{1}{\tau} \mathbb{X}_n^{r_n+s,r_n+t} \given[\Big] \norm{\mathbf{X}_{r_n}} > u_n(\tau)\right)$ which is, for $j=s,\dots,t$, the joint distribution of the unit vectors in the directions of $\mathbf{X}_{r_n+j}$ scaled by the ratios $\dfrac{u_n^{-1}(\norm{\mathbf{X}_{r_n+j}})}{\tau}$ conditional on $u_n^{-1}(\norm{\mathbf{X}_{r_n}}) \leq \tau$. Then, $(Z_j)_{j \in \mathbb{Z}}$ distributes as $(Y_j)_{j \in \mathbb{Z}}$ conditional on $u_n^{-1}(\norm{\mathbf{X}_{r_n+j}})>\tau$ for all negative $j$.

In the remainder of this subsection we include some definitions that we will need later on.

Let $\tilde{l}_{\infty}=l_{\infty}/_{\sim}$ where, for all $\mathbf{x},\mathbf{y} \in l_{\infty}$, $\mathbf{x} \sim \mathbf{y} \iff \mathbf{y}=\sigma^k(\mathbf{x})$ for some $k \in \mathbb{Z}$ ($\sigma$ being the left-shift map) and let $\tilde{\pi}:l_{\infty} \to \tilde{l}_{\infty}$ be the canonical projection.

For the computation of the distribution of the piling process we may use its polar decomposition, that is $L_Z:=\displaystyle\inf_{j \in \mathbb{Z}} \norm{Z_j}$ and $Q_j:=\dfrac{Z_j}{L_Z}$ for all $j \in \mathbb{Z}$.

Let
\begin{equation*}
l_{0}=\{\mathbf{x}=(x_j)_{j \in \mathbb{Z}} \in (\mathbb{R}^d)^{\mathbb{Z}}: \lim_{\modl{j} \to \infty} \norm{x_j}=0\}.
\end{equation*}
Let
\begin{equation*}
\func{p}{\overline{\mathbb{R}^d}\setminus\{0\}}{\mathbb{R}^d}{x}{\begin{cases}
\frac{x}{\norm{x}^2}, \ x \neq \infty\\
0, \ \text{otherwise}
\end{cases}}
\end{equation*}
and $P:l_{\infty} \to l_{0}$ defined as $P((x_j)_{j \in \mathbb{Z}})=(p(x_j))_{j \in \mathbb{Z}}$. For $\tilde{l}_0$ defined in the same way as $\tilde{l}_{\infty}$ above, then $P:\tilde{l}_{\infty} \to \tilde{l}_{0}$ is such that $\tilde{P}(\tilde{\pi}(\mathbf{x}))=\tilde{\pi}(P(\mathbf{x}))$. Let
\begin{equation}\label{eq:psi-aux}
\func{\psi}{\tilde{l}_{\infty} \setminus \{\tilde{\infty}\}}{\mathbb{R}^{+} \times \mathbb{S}}{\tilde{\mathbf{x}}}{\left(\dfrac{1}{\norm{\tilde{P}(\tilde{\mathbf{x}})}_{\infty}},\dfrac{\tilde{\mathbf{x}}}{\norm{\tilde{P}(\tilde{\mathbf{x}})}_{\infty}^{-1}}\right)}.
\end{equation}
Finally, let $\mathbf{\tilde{Q}}=\tilde{\pi}((Q_j)_{j \in \mathbb{Z}})$.
\begin{remark}\label{rmk:piling-distr}
Notice that $\psi(\tilde{\pi}((Z_j)_{j \in \mathbb{Z}}))=(L_Z,\mathbf{\tilde{Q}})$. $L_Z$ is uniformly distributed on $[0,1]$ and independent of $\mathbf{\tilde{Q}}$ (by \cite[Corollary 3.14]{20enriched}) whose distribution is $\mathbb{P}_{\mathbf{\tilde{Q}}}$.
\end{remark}
Let
\begin{equation}\label{eq:xi-map}
\func{\xi}{\overline{\mathbb{R}^d}\setminus\{0\}}{\mathbb{R}^{d}}{x}{\begin{cases}
(\norm{x})^{-\frac{1}{\alpha}}\frac{x}{\norm{x}}, \ x \neq \infty\\
0, \ \text{otherwise}
\end{cases}}.
\end{equation}

\subsection{Cluster point process}\label{subsec:cluster-pp}
Let the process $(\mathbf{X}_n)_{n \in \mathbb{N}_0}$ run until time $n$ when $k_n$ blocks are each made of $r_n:=\floor{n/k_n}$ observations. This means sequences $(k_n)_{n \in \mathbb{N}}$ and $(r_n)_{n \in \mathbb{N}}$ which, additionally, must be such that $k_n,r_n \underset{\scriptsize{n \to \infty}}{\longrightarrow}\infty$. We note that this is a declustering method in the sense that we determine that observations belonging to different blocks correspond to different clusters (see Section 3.1 of \cite{20enriched} for more details).

Then, we write
\begin{equation*}
\mathbb{X}_{n,i}\equiv \mathbb{X}_{n}^{r_n(i-1),r_n.i-1}=\left(u_n^{-1}(\norm{\mathbf{X}_{r_n(i-1)}})\dfrac{\mathbf{X}_{r_n(i-1)}}{\norm{\mathbf{X}_{r_n(i-1)}}},\dots,u_n^{-1}(\norm{\mathbf{X}_{r_n.i-1}})\dfrac{\mathbf{X}_{r_n.i-1}}{\norm{\mathbf{X}_{r_n.i-1}}}\right)
\end{equation*}
so that the unit vector in the direction of $\mathbf{X}_j$ is scaled by the asymptotic frequency $u_n^{-1}(\norm{\mathbf{X}_j})$ for all indices $j$ matching those of the observations in the $i$-th block.
\begin{remark}\label{rmk:X_n,s,t-l-inf}
Observe that there is a natural identification between $\mathbb{X}_n^{s,t}$ and an element of $l_{\infty}$ simply by considering, in $l_{\infty}$, the bi-infinite sequence with all entries to the left of $u_n^{-1}(\norm{\mathbf{X}_{s}})\dfrac{\mathbf{X}_{s}}{\norm{\mathbf{X}_{s}}}$ and all entries to the right of $u_n^{-1}(\norm{\mathbf{X}_{t}})\dfrac{\mathbf{X}_{t}}{\norm{\mathbf{X}_{t}}}$ equal to $\infty$ (and $\mathbb{X}_n^{s,t}$ in between).
\end{remark}

The (empirical) cluster point process
\begin{equation}\label{eq:gen-N_n}
N_n=\sum_{i=1}^{\infty} \delta_{(i/k_n,\tilde{\pi}(\mathbb{X}_{n,i}))}
\end{equation}
places the pile $\mathbb{X}_{n,i}$ (seen as a bi-infinite sequence, under the identification given in Remark \ref{rmk:X_n,s,t-l-inf}) on the vertical line through $i/k_n$, with $\tilde{\pi}$ allowing for the assumption that the entries $u_n^{-1}(\norm{\mathbf{X}_{r_n(i-1)}})\dfrac{\mathbf{X}_{r_n(i-1)}}{\norm{\mathbf{X}_{r_n(i-1)}}}$ are placed in the horizontal axis.
\begin{theorem}[\cite{20enriched}]\label{thm:pp-conv-FFT20}
If a normalising sequence $(u_n)_{n \in \mathbb{N}}$ as above exists, the piling process is well defined and conditions $\D_{q_n}$ and $\D'_{q_n}$ (as in Appendix \ref{appx:dependance}) are satisfied, then $N_n$ converges weakly (in the space of boundedly finite point measures on $\mathbb{R}_0^{+} \times \tilde{l}_{\infty} \setminus \{\tilde{\infty}\}$ with weak$^{\#}$ topology) to
\begin{equation}\label{eq:gen-N}
N=\sum_{i=1}^{\infty} \delta_{(T_i,U_i\mathbf{\tilde{Q}_i)}}
\end{equation}
which is a Poisson point process with intensity measure $Leb \times \eta$, where $\eta=\vartheta(Leb \times \mathbb{P}_{\mathbf{\tilde{Q}}}) \circ \psi$ (\textit{cf.} (\ref{eq:psi-aux})).
\end{theorem}

\subsection{Enriched functional limit theorem}
\begin{definition}\label{def:alpha-reg}
$(\mathbf{X}_n)_{n \in \mathbb{N}_0}$ has an \textit{$\alpha$-regularly varying tail}, $\alpha \in (0,2)$, if there exists a sequence $(a_n)_{n \in \mathbb{N}}$ of positive real numbers such that
\begin{equation}\label{eq:alpha-hea}
\lim_{n \to \infty} n\mathbb{P}(\norm{\mathbf{X}_0}>ya_n)=y^{-\alpha}.
\end{equation}
\end{definition}
\begin{remark}
(\ref{eq:alpha-hea}) implies (\ref{eq:u_n-tau}): letting $\tau=y^{-\alpha}$ we have $u_n(\tau)=\tau^{-\frac{1}{\alpha}}a_n$. Moreover, $u_n^{-1}(z)=\left(\dfrac{z}{a_n}\right)^{-\alpha}$.
\end{remark}
\begin{theorem}[\cite{20enriched}]\label{thm:func-lim-FFT20}
Let $(\mathbf{X}_n)_{n \in \mathbb{N}_0}$ have an $\alpha$-regularly varying tail, where $\alpha \in (0,1)$. Under the hypotheses of Theorem \ref{thm:pp-conv-FFT20},
\begin{equation*}
S_n(t):=\sum_{i=0}^{\floor{nt}-1} \dfrac{1}{a_n}\mathbf{X}_i-tc_n, \ t \in [0,1],
\end{equation*}
converges in $F'$ to $(V,\text{disc}(V),(e_V^s)_{s \in \text{disc}(V)})$, where $V$ is an $\alpha$-stable Lévy process on $[0,1]$
\begin{equation*}
V(t)=\sum_{T_i \leq t}\sum_{j \in \mathbb{Z}} U_i^{-\frac{1}{\alpha}}\mathcal{Q}_{i,j}, \ t \in [0,1],
\end{equation*}
where $T_i$ and $U_i$ are as in $N$ in Theorem \ref{thm:pp-conv-FFT20}, $\mathcal{Q}_i=\xi(\mathbf{\tilde{Q}}_i)$ for $\mathbf{\tilde{Q}}_i$ as in $N$ in Theorem \ref{thm:pp-conv-FFT20} and $\xi$ as in (\ref{eq:xi-map}) and the excursions are given by $e_V^{T_i}(t)=V(T_i^{-})+U_i^{-2}\displaystyle\sum_{0 \leq j \leq \floor{\tan(\pi(t-\frac{1}{2}))}} \mathcal{Q}_{i,j}$.
\end{theorem}
We make a small comment on the proof of Theorem \ref{thm:func-lim-FFT20} (which is detailed in section 4.1 of \cite{20enriched}). The convergence of the empirical cluster point process $N_n$ expressed in Theorem \ref{thm:pp-conv-FFT20}  is the crucial preliminary step to Theorem \ref{thm:func-lim-FFT20}. By the Continuous Mapping Theorem, such point process convergence obtained in the dynamical setting translates to a point process convergence in the pure probabilistic setting of \cite{basrak18}. Then, proving convergence in $F'$ means proving convergence of the projections into two other functional spaces $E$ and $\tilde{D}$. The convergence in $E$ follows from \cite{basrak18} and the convergence in $\tilde{D}$ is shown in \cite[Proposition 4.4]{20enriched}. We note that we are restricting to the case of $\alpha \in (0,1)$.

Thus, we must check that equation (\ref{eq:alpha-hea}) holds in our examples. As for the conditions $\D_{q_n}$ and $\D'_{q_n}$, they are usual dependence requirements that can be deduced from good mixing properties of the underlying dynamical systems. We do recall them and verify that they hold for the examples we consider, but leave that work to Appendix \ref{appx:dependance} (and \ref{appx:dependance-ex}). So, the remaining work in obtaining Theorem \ref{thm:func-lim-FFT20} is to deal with the piling processes.

\section{A finite number of points in the same orbit}\label{sec:corr-finite}
Here $\mathcal{M}=\{\xi_1,\dots,\xi_N\}$ such that there exist $m_1,\dots,m_N$ with $\xi_i=f^{m_i}(\zeta)$, where $\zeta \in \mathcal{X}$, and we take $m_1=0$ (\textit{i.e.} $\xi_1=\zeta$). The piling processes are fundamentally different when $\zeta$ is not/is periodic, as can be seen in Theorem \ref{thm:piling-corr-nper} and in Theorem \ref{thm:piling-corr-per}, respectively. We illustrate the use of the same theorems with examples inspired in those in sections 4.3 and 5.3 of \cite{16correlated}.

Before presenting the statements and proofs of Theorems \ref{thm:piling-corr-nper} and \ref{thm:piling-corr-per} we provide some heuristics and fix notation.

Let an exceedance of the level $u_n(\tau)$, at time $r_n$, be due to a hit to a specified small neighbourhood of $\xi_i$ for a certain $i \in \{1,\dots,N\}$ that we now fix. Then, at times $r_n+j$ where $j=m_l-m_i$ for all $l=1,\dots,i-1,i+1,\dots,N$, there are hits to neighbourhoods of $f^j(\xi_i)$, which correspond to exceedances of levels $u_n(\tau_1),\dots,u_n(\tau_{i-1}),u_n(\tau_{i+1}),\dots,u_n(\tau_N)$. The piling process stores, at position $j=m_l-m_i$ for all $l=1,\dots,i-1,i+1,\dots,N$, (the asymptotic behaviour of) the ratio $\dfrac{\tau_l}{\tau}$ projected on the direction of the point whose observation exceeds $u_n(\tau_l)$. As we prove further on, this is (asymptotically) tied to the derivative of $f$ and the constants $c_i$ appearing in the $h_i$ (recall (\ref{eq:hi})).

In addition, the definition of the piling process requires that all negatively indexed entries have norms greater than or equal to $1$ (\textit{cf.} Definition \ref{def:piling}). When the $c_i$ are all equal and $f$ is expanding, then all negatively indexed entries must be equal to $\infty$. However, when the $c_i$ differ, whether it is possible to have negatively indexed entries with norms greater than or equal to $1$ (and that are not $\infty$) depends on the balance between $\left(\dfrac{c_i}{c_l}\right)^{\alpha}$ and the norm of $Df_{\xi_i}^{m_l-m_i}$, for all $l=1,\dots,i-1$. One can think in terms of how much the factor $\left(\dfrac{c_i}{c_l}\right)^{\alpha}$ makes a contraction look like an expansion - we use the expression `fake expansion' to refer to this.

We are going to restrict to the cases where fake expansion, if it exists, holds at every direction. We remark that we can always have that for a suitable choice of the constants $c_i$ (depending on the derivative of $f$). Although we miss full generality with such an assumption, we do it to avoid an otherwise very technical statement.

The set $A^{(i)}$, that we define next, stores the indices $j=m_l-m_i$, for all $l \in \{1,\dots,i-1\}$, compatible with entries different from $\infty$ that appear left of index $j=0$ (indices compatible with fake expansion).

Let $\mathcal{I}=\{1,\dots,N\}$ (recall from Section \ref{sec:intro} that $\mathcal{M}=\{\xi_1,\dots,\xi_N\}$ corresponds to $\mathcal{I}=\{1,\dots,N\}$).

We split into two cases according to $\zeta$ (\textit{i.e.} $\xi_1$) being non-periodic or periodic, respectively.
\begin{enumerate}[label=(A\arabic*)]
\item\label{itm:A1} Assume $\zeta$ is non-periodic. Let $i \in \mathcal{I}$. If $i>1$, for all $l=1,\dots,i-1$, let
\begin{equation*}
\lambda_{i,l}^{min}=\min\{\norm{Df_{\xi_i}^{m_l-m_i}(w)}: \norm{Df_{\xi_i}^{m_l-m_i}(w)}<1, w \in \mathbb{S}^{d-1}\},
\end{equation*}
\begin{equation*}
\lambda_{i,l}^{max}=\max\{\norm{Df_{\xi_i}^{m_l-m_i}(w)}: \norm{Df_{\xi_i}^{m_l-m_i}(w)}<1, w \in \mathbb{S}^{d-1}\},
\end{equation*}
and define $u_{i,l}^{min}=\left(\dfrac{c_l}{c_i}\right)^{\alpha}\dfrac{1}{\lambda_{i,l}^{min}}$ and $u_{i,l}^{max}=\left(\dfrac{c_l}{c_i}\right)^{\alpha}\dfrac{1}{\lambda_{i,l}^{max}}$. Let $A^{(1)}:=\emptyset$ and, for all $i>1$, $A^{(i)}:=\{m_l-m_i: u_{i,l}^{min} < 1\}$.
\item\label{itm:A2} Assume $\zeta$ is periodic of prime period $q$. Let $i \in \mathcal{I}$. For all $i,l \in \mathcal{I}$, if $s \in \mathbb{N}_0$ is such that $m_l-m_i-qs<0$, let
\begin{equation*}
\lambda_{i,l,s}^{min}=\min\{\norm{Df_{\xi_i}^{m_l-m_i-qs}(w)}: \norm{Df_{\xi_i}^{m_l-m_i-qs}(w)}<1, w \in \mathbb{S}^{d-1}\},
\end{equation*}
\begin{equation*}
\lambda_{i,l,s}^{max}=\max\{\norm{Df_{\xi_i}^{m_l-m_i-qs}(w)}: \norm{Df_{\xi_i}^{m_l-m_i-qs}(w)}<1, w \in \mathbb{S}^{d-1}\},
\end{equation*}
and define $u_{i,l,s}^{min}=\left(\dfrac{c_l}{c_i}\right)^{\alpha}\dfrac{1}{\lambda_{i,l,s}^{min}}$ and $u_{i,l,s}^{max}=\left(\dfrac{c_l}{c_i}\right)^{\alpha}\dfrac{1}{\lambda_{i,l,s}^{max}}$.\\
Let $A^{(i)}:=\{m_l-m_i-qs: u_{i,l,s}^{min} < 1\}$.
\end{enumerate}
We give a visual interpretation of how $A^{(i)}$ is obtained in case \ref{itm:A1} in Figure \ref{fig:A-i}. In case \ref{itm:A2} the reasoning is analogous accounting for the successive hits to neighbourhoods of the points in $\mathcal{M}$ at every integer multiple of the period $q$.
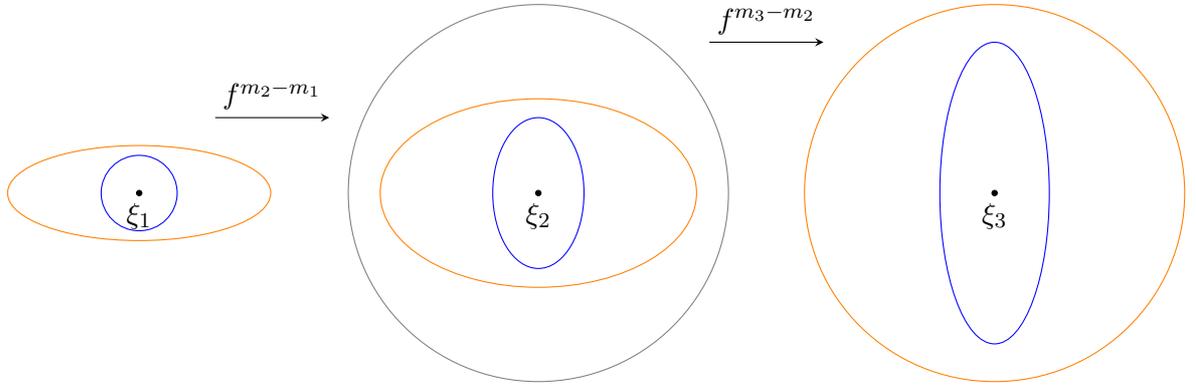
\begin{figure}[H]
\centering
\begin{tikzpicture}
\filldraw[black] (0,0) circle (1pt);
\node at (0,-0.3) {$\xi_1$};
\draw[thin,blue] (0,0) circle (0.5cm);
\draw[thin,orange] (0,0) ellipse[x radius=1.73cm, y radius=0.63cm];
\filldraw[black] (5.25,0) circle (1pt);
\node at (5.25,-0.3) {$\xi_2$};
\draw[thin,gray] (5.25,0) circle (2.5cm);
\draw[thin,blue] (5.25,0) ellipse[x radius=0.6cm, y radius=1cm];
\draw[thin,orange] (5.25,0) ellipse[x radius=2.08cm, y radius=1.25cm];
\filldraw[black] (11.25,0) circle (1pt);
\node at (11.25,-0.3) {$\xi_3$};
\draw[thin,orange] (11.25,0) circle (2.5cm);
\draw[thin,blue] (11.25,0) ellipse[x radius=0.72cm, y radius=2cm];
\draw [-stealth](1,1) -- (2.5,1);
\node at (1.75,1.3) {$f^{m_2-m_1}$};
\draw [-stealth](7.5,2) -- (9,2);
\node at (8.25,2.3) {$f^{m_3-m_2}$};
\end{tikzpicture}
\caption{Example with $\mathcal{I}=\{1,2,3\}$, $i=3$ and $A^{(3)}=\{m_1-m_3\}$, for a $2$-dimensional uniformly expanding $f$ and observable $\Psi$ as in (\ref{eq:psi}) with $c_1$ sufficiently smaller than $c_3$ so that fake expansion holds at $j=m_1-m_3$, and $c_2=c_3$ (\textit{cf.} (\ref{eq:hi})). The boundaries of the balls of radius $h_i^{-1}(u_n(\tau))$ centred at $\xi_i$ for $i=1,2,3$ are the blue, grey and orange circles, respectively. We may say blue, grey or orange ball to refer to, respectively, the ball around $\xi_1$, $\xi_2$ or $\xi_3$ whose boundary is the blue, grey or orange circle. We condition on an exceedance of the threshold $u_n(\tau)$ being due to a hit (at time $j=0$) to the orange ball. Then, at times $j=m_2-m_3$ and $j=m_1-m_3$ there may have been hits to some neighbourhoods of $\xi_2$ and $\xi_1$, respectively. Because the connected component of the $(m_3-m_2)$-th pre-image of the orange ball which intersects the grey ball is strictly contained in the grey ball, we have $u_{3,2}^{max}>1$ (which implies $u_{3,2}^{min}>1$). On the other hand, the connected component of the $(m_3-m_1)$-th pre-image of the orange ball which intersects the blue ball strictly contains the blue ball which is compatible with $u_{3,1}^{min}<1$. This leads to $A^{(3)}=\{m_1-m_3\}$. Thus, the piling process has (a.s.) an infinity entry at position $j=m_2-m_3$; however, there exists a region around $\xi_3$ (the annulus delimited by the blue ellipse and the orange circle) for which the piling process has non-infinity entries at position $j=m_1-m_3$. We note that the piling process still has infinity entries at position $j=m_1-m_3$ if the hit to the orange ball around $\xi_3$ belongs to the interior of the blue ellipse, or if it is preceded by a hit (at time $m_1-m_3$) to a neighbourhood of a pre-image of $\xi_3$ which is not $\xi_1$.}
\label{fig:A-i}
\end{figure}
\begin{remark}
The condition $u_{i,l}^{min} < 1$, where $i \in \mathcal{I}$ and $l \in \{1,\dots,i-1\}$, expresses the geometrical requirement that, for all $n \in \mathbb{N}$, the image under $f^{m_i-m_l}$ of the ball of radius $h_l^{-1}(u_n(\tau))$ around $\xi_l$ is strictly contained in the ball of radius $h_i^{-1}(u_n(\tau))$ around $\xi_i$. In Figure \ref{fig:A-i}, the blue ball around $\xi_1$ is mapped by $f^{m_3-m_1}$ to an ellipse which is strictly contained in the orange ball around $\xi_1$.
\end{remark}

Moreover, the balance between $\left(\dfrac{c_i}{c_l}\right)^{\alpha}$ and the norm of $Df_{\xi_i}^{m_l-m_i}$, for all $l=1,\dots,i-1$, also determines the probabilities with which the different sequences corresponding to the indices $m_l-m_i \in A^{(i)}$ appear - due to the $u_{i,l}^{min},u_{i,l}^{max}$ determining ranges for each sequence. This is formalised in the statements of our theorems.

\subsection{Main theorem and examples in the non-periodic case}
\begin{theorem}\label{thm:piling-corr-nper}
Let $f$, $\mu$ and $\Psi$ be as specified in Section \ref{sec:intro} with $\mathcal{M}$ as in \ref{itm:A}, where $\zeta$ is a non-periodic point. Assume that $(\mathbf{X}_n)_{n \in \mathbb{N}_0}$ has an $\alpha$-regularly varying tail, where $\alpha \in (0,1)$. For all $i \in \mathcal{I}$, define $p_i=\dfrac{D_ic_i^{\alpha d}}{\sum_{k=1}^{N}D_kc_k^{\alpha d}}$. Let $A^{(i)}$ be as defined in \ref{itm:A1}. If $A^{(i)}=\emptyset$, the piling process is
\begin{enumerate}[(0)]
\item\label{itm:(0)-piling-fin-nper} with probability $p_i$ the bi-infinite sequence $(Z_j)_{j \in \mathbb{Z}}$ with:
\begin{enumerate}[(i)]
\item entry $U.\Theta$ at $j=0$;
\item entries $U.Df_{\xi_i}^{m_l-m_i}(\Theta)\left(\dfrac{c_i}{c_l}\right)^{\alpha}$ at $j=m_l-m_i$ for all $l=i+1,\dots,N$;
\item $\infty$ for all other positive indices $j$;
\item $\infty$ for all negative indices $j$;
\end{enumerate}
where $U$ is uniformly distributed on $[0,1]$, $\Theta$ is uniformly distributed on $\mathbb{S}^{d-1}$, and $U$ and $\Theta$ are independent.
\end{enumerate}
If $A^{(i)} \neq \emptyset$, assume there exists an increasing ordering of the $u_{i,l}^{min,max}$ such that $u_{i,l_p}^{min} \leq u_{i,l_{p+1}}^{max}$ for all $p \in \{1,\dots,\#A^{(i)}-1\}$. Then, the piling process is
\begin{enumerate}[(I)]
\item with probability $p_iu_{i,l_1}^{max}$ the bi-infinite sequence $(Z_j)_{j \in \mathbb{Z}}$ with:\begin{enumerate}[(i)]
\item entry $U_0.\Theta$ at $j=0$;
\item entries $U_0.Df_{\xi_i}^{m_l-m_i}(\Theta)\left(\dfrac{c_i}{c_l}\right)^{\alpha}$ at $j=m_l-m_i$ for all $l=i+1,\dots,N$;
\item $\infty$ for all other positive indices $j$;
\item $\infty$ for all negative indices $j$;
\end{enumerate}
where $U_0$ is uniformly distributed on $[0,u_{i,l_1}^{max})$, $\Theta$ is uniformly distributed on $\mathbb{S}^{d-1}$, and $U_0$ and $\Theta$ are independent;
\item\label{itm:(II)-piling-fin-nper} with probability $p_i(u_{i,l_p}^{min}-u_{i,l_p}^{max})$, where $p \in \{1,\dots,\#A^{(i)}\}$, the bi-infinite sequence $(Z_j)_{j \in \mathbb{Z}}$ with:\begin{enumerate}[(i)]
\item entry $U_{p'}.\Theta$ at $j=0$;
\item entries $U_{p'}.Df_{\xi_i}^{m_l-m_i}(\Theta)\left(\dfrac{c_i}{c_l}\right)^{\alpha}$ at $j=m_l-m_i$ for all $l=i+1,\dots,N$;
\item $\infty$ for all other positive indices $j$;
\item entries $U_{p'}.Df_{\xi_i}^{m_l-m_i}(\Theta)\left(\dfrac{c_i}{c_l}\right)^{\alpha}$ at $j=m_l-m_i$ for all $l \in \{l_1,\dots,l_p\}$;
\item $\infty$ for all other negative indices $j$;
\end{enumerate}
where $U_{p'}$ is uniformly distributed on $[u_{i,l_p}^{max},u_{i,l_p}^{min})$ and $\Theta \given \{U_{p'}=u\}$ is uniformly distributed on $\left\lbrace w \in \mathbb{S}^{d-1}:\norm{Df_{\xi_i}^{m_{l_p}-m_i}(w)} \geq \dfrac{1}{u}\left(\dfrac{c_{l_p}}{c_i}\right)^{\alpha} \right\rbrace$;
\item with probability $p_i(u_{i,l_{p+1}}^{max}-u_{i,l_p}^{min})$, where $p \in \{1,\dots,\#A^{(i)}-1\}$, the bi-infinite sequence $(Z_j)_{j \in \mathbb{Z}}$ with:\begin{enumerate}[(i)]
\item entry $U_p.\Theta$ at $j=0$;
\item entries $U_p.Df_{\xi_i}^{m_l-m_i}(\Theta)\left(\dfrac{c_i}{c_l}\right)^{\alpha}$ at $j=m_l-m_i$ for all $l=i+1,\dots,N$;
\item $\infty$ for all other positive indices $j$;
\item entries $U_p.Df_{\xi_i}^{m_l-m_i}(\Theta)\left(\dfrac{c_i}{c_l}\right)^{\alpha}$ at $j=m_l-m_i$ for all $l \in \{l_1,\dots,l_p\}$;
\item $\infty$ for all other negative indices $j$;
\end{enumerate}
where $U_p$ is uniformly distributed on $[u_{i,l_p}^{min},u_{i,l_{p+1}}^{max})$, $\Theta$ is uniformly distributed on $\mathbb{S}^{d-1}$, and $U_p$ and $\Theta$ are independent;
\item with probability $p_i(1-u_{i,l_{\#A^{(i)}}}^{min})$ the bi-infinite sequence $(Z_j)_{j \in \mathbb{Z}}$ with:\begin{enumerate}[(i)]
\item entry $U_{\#A^{(i)}}.\Theta$ at $j=0$;
\item entries $U_{\#A^{(i)}}.Df_{\xi_i}^{m_l-m_i}(\Theta)\left(\dfrac{c_i}{c_l}\right)^{\alpha}$ at $j=m_l-m_i$ for all $l=i+1,\dots,N$;
\item $\infty$ for all other positive indices $j$;
\item entries $U_{\#A^{(i)}}.Df_{\xi_i}^{m_l-m_i}(\Theta)\left(\dfrac{c_i}{c_l}\right)^{\alpha}$ at $j=m_l-m_i$ for all $l \in A^{(i)}$;
\item $\infty$ for all other negative indices $j$;
\end{enumerate}
where $U_{\#A^{(i)}}$ is uniformly distributed on $[u_{i,l_{\#A^{(i)}}}^{min},1]$, $\Theta$ is uniformly distributed on $\mathbb{S}^{d-1}$, and $U_{\#A^{(i)}}$ and $\Theta$ are independent.
\end{enumerate}
\end{theorem}
\begin{remark}\label{rmk:equal-cs}
Assume $f$ is expanding and, for all $i \in \mathcal{I}$, $c_i=c$. Then, for any $i \in \mathcal{I}$, $\norm{Df_{\xi_i}^{m_l-m_i}(w)}<1$ for all $l=1,\dots,i-1$ and for all $w \in \mathbb{S}^{d-1}$. It follows that $A^{(i)}=\emptyset$ for all $i \in \mathcal{I}$. In particular, the piling process will be of the simplest form given by case \ref{itm:(0)-piling-fin-nper} in Theorem \ref{thm:piling-corr-nper} (where, in (ii), $\left(\dfrac{c_i}{c_l}\right)^{\alpha}=1$ implies the reduction to $U.Df_{\xi_i}^{m_l-m_i}(\Theta)$).
\end{remark}
\begin{remark}\label{rmk:equal-u-maxmin-1}
If $f$ is 1-dimensional then $u_{i,l}^{min}=u_{i,l}^{max}$ and case \ref{itm:(II)-piling-fin-nper} in Theorem \ref{thm:piling-corr-nper} doesn't exist.
\end{remark}
\begin{remark}
When $A^{(i)} \neq \emptyset$, the requirement for an increasing ordering of the $u_{i,l}^{min,max}$ such that $u_{i,l_p}^{min} \leq u_{i,l_{p+1}}^{max}$ for all $p \in \{1,\dots,\#A^{(i)}-1\}$ means a picture like Figure \ref{fig:u-i}. We observe that our choice of $c_i$ for $h_i$ (\textit{cf.} (\ref{eq:hi})), besides being tied to fake expansion, determines the existence of such increasing ordering. When $f$ is 1-dimensional it reduces to $u_{i,l_p} \leq u_{i,l_{p+1}}$, for all $p \in \{1,\dots,\#A^{(i)}-1\}$, which is trivially satisfied.
\begin{figure}[H]
\centering
\begin{tikzpicture}
\filldraw[black] (0,0) circle (1pt);
\node at (0,-0.3) {$\xi_1$};
\draw[thin] (0,0) circle (0.75cm);
\draw[thick,dotted] (0,0) ellipse[x radius=1.25cm, y radius=0.75cm];
\draw[thick,dotted] (0,0) ellipse[x radius=0.75cm, y radius=0.25cm];
\filldraw[black] (5.25,0) circle (1pt);
\node at (5.25,-0.3) {$\xi_2$};
\draw[thin] (5.25,0) circle (1.25cm);
\draw[dashed] (5.25,0) ellipse[x radius=1.75cm, y radius=1.25cm];
\draw[dashed] (5.25,0) ellipse[x radius=1.25cm, y radius=0.75cm];
\filldraw[black] (11.25,0) circle (1pt);
\node at (11.25,-0.3) {$\xi_3$};
\draw[thin] (11.25,0) circle (2cm);
\draw[thick,dotted] (11.25,0) circle (1.75cm);
\draw[thick,dotted] (11.25,0) circle (1.5cm);
\draw[dashed] (11.25,0) circle (1cm);
\draw[dashed] (11.25,0) circle (0.75cm);
\end{tikzpicture}
\caption{Example with $\mathcal{I}=3$, $i=3$ and $A^{(3)}=\{1,2\}$, for a $2$-dimensional uniformly expanding $f$ and observable $\Psi$ as in (\ref{eq:psi}) with $c_1$ sufficiently smaller than $c_2$ and $c_2$ sufficiently smaller than $c_3$ so that fake expansion holds at both $j=m_1-m_3$ and $j=m_2-m_3$ and, additionally, $u_{3,l_1}^{min} \leq u_{3,l_{2}}^{max}$ with $l_1,l_2 \in \{1,2\}$. The boundaries of the balls of radius $h_i^{-1}(u_n(\tau))$ centred at $\xi_i$ for $i=1,2,3$ are the three black circles. We condition on an exceedance of the threshold $u_n(\tau)$ being due to a hit (at time $j=0$) to the ball of radius $h_3^{-1}(u_n(\tau))$ around $\xi_3$. The dotted annulus around $\xi_1$ (\textit{i.e.} the annulus delimited by the dotted ellipses around $\xi_1$) is mapped by $f^{m_3-m_1}$ to the dotted annulus around $\xi_3$ (\textit{i.e.} the annulus delimited by the dotted circles around $\xi_3$). Analogous statement for the dashed annuli around $\xi_2$ and $\xi_3$ and $f^{m_3-m_2}$. The interior of the smaller dotted circle around $\xi_3$ is compatible with infinity entries, in every direction, for the piling process at position $m_1-m_3$, while the exterior of the bigger dotted circle around $\xi_3$ is compatible with non-infinity entries, in every direction, for the piling process at position $j=m_1-m_3$ (\textit{cf.} Figure \ref{fig:A-i}). Analogous statement holds for the dashed annuli and entries at position $j=m_2-m_3$. The requirement $u_{3,l_1}^{min} \leq u_{3,l_{2}}^{max}$ means that the dotted and dashed annuli around $\xi_3$ are disjoint. Then, a non-infinity entry in any chosen direction at position $j=m_1-m_3$ implies a non-infinity entry in every direction at position $j=m_2-m_3$. In particular, $l_1=2$ and $l_2=1$.}
\label{fig:u-i}
\end{figure}
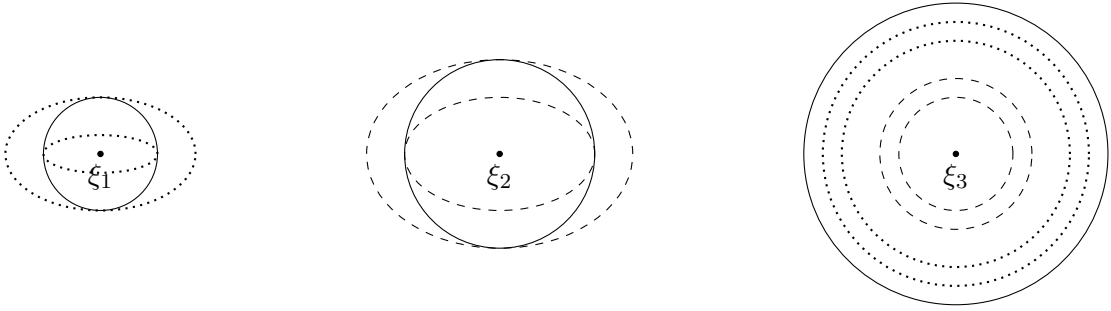
\end{remark}
We illustrate simple cases of Theorem \ref{thm:piling-corr-nper} by Examples \ref{ex:nper-evendist} and \ref{ex:nper-weighted}. After these, we provide an example of application of the general version of the theorem.
\begin{example}\label{ex:nper-evendist}
Let $f(x)=2x \mod 1$, $x \in [0,1]$, and $\mu= $ Lebesgue measure on $[0,1]$ (invariant for $f$). Take $\zeta=\dfrac{\sqrt{2}}{16}$ (non-periodic), and define the observable $\psi$ as
\begin{equation*}
\psi(x):=\begin{cases}
\modl{x-\zeta}^{-2}, \ x \in B_{\varepsilon_1}(\zeta)\\
\modl{x-f(\zeta)}^{-2}, \ x \in B_{\varepsilon_2}(f(\zeta))\\
\modl{x-f^3(\zeta)}^{-2}, \ x \in B_{\varepsilon_3}(f^3(\zeta))\\
0, \ \text{otherwise}
\end{cases}
\end{equation*}
for some $\varepsilon_1,\varepsilon_2,\varepsilon_3>0$. Observe that, presented as in (\ref{eq:psi}),
\begin{equation*}
\psi(x)=\displaystyle\sum_{i=1}^{3} h_i(\modl{x-\xi_i})\mathbf{1}_{B_{\varepsilon_i}(\xi_i)}(x)
\end{equation*}
for $\xi_1=\zeta$, $\xi_2=f(\zeta)$, $\xi_3=f^3(\zeta)$, and $h_i(t)=t^{-2}$ for $i=1,2,3$ (so that $\alpha=1/2$). In particular, $\mathcal{M}=\{\zeta,f(\zeta),f^3(\zeta)\}$ and equation (\ref{eq:alpha-hea}) holds with $a_n=36n^2$.

Since $\mu=\text{Lebesgue}$, we have that $D_i=1$ for $i=1,2,3$. Also, $d=1$ and $c_i=1$ for $i=1,2,3$. Thus, $p_1=p_2=p_3=\dfrac{1}{3}$. We have $f'(x)=2$ for all $x \in [0,1]$, so that $(f^{-j})'(\xi_i)=\dfrac{1}{2^j}$ when $j<0$ for $i=1,2,3$ leading to $A^{(i)}=\emptyset$ for $i=1,2,3$. Applying Theorem \ref{thm:piling-corr-nper}, we conclude that the piling process
is any of the bi-infinite sequences $(Z_j)_{j \in \mathbb{Z}}$
$$\left(\dots,\infty,U,U.2,\infty,U.2^3,\infty,\dots\right)$$
$$\left(\dots,\infty,U,\infty,U.2^2,\infty,\dots\right)$$
$$\left(\dots,\infty,U,\infty,\dots\right)$$
with probability $\dfrac{1}{3}$, where $U$ is uniformly distributed on $[0,1]$.

We clarify the notation used by stressing that, in all the three sequences presented, the entries equal to $U$ correspond to index $j=0$ and the entries that are visibly different from $\infty$ are the only such entries, as is imposed by the statement of Theorem \ref{thm:piling-corr-nper}.
\end{example}
\begin{remark}\label{rmk:EI-ex-1}
For Example \ref{ex:nper-evendist}, we have
\begin{equation*}
\mu(U_n(\tau))=\displaystyle\sum_{i=1}^{3} \mu(B_{h_i^{-1}(u_n(\tau))}(\xi_i))=3.2u_n(\tau)^{-\frac{1}{2}}.
\end{equation*}
Let $q_n=m_3-m_1=3$. It follows from \cite[Corollary 4.5]{16correlated},
\begin{equation*}
\begin{split}
\mu(U_n^{(q_n)}(\tau))&=\mu(B_{h_1^{-1}(u_n(\tau))}(\xi_1))-\frac{1}{2}\mu(B_{h_2^{-1}(u_n(\tau))}(\xi_2))\\
&\ +\mu(B_{h_2^{-1}(u_n(\tau))}(\xi_2))-\frac{1}{2^2}\mu(B_{h_3^{-1}(u_n(\tau))}(\xi_3))\\
&\ +\mu(B_{h_3^{-1}(u_n(\tau))}(\xi_3))\\
&=2u_n(\tau)^{-\frac{1}{2}}(1-\frac{1}{2}+1-\frac{1}{2^2}+1).
\end{split}
\end{equation*}
Thus, the extremal index is
\begin{equation*}
\begin{split}
\vartheta=\lim_{n \to \infty} \dfrac{\mu(U_n^{(q_n)}(\tau))}{\mu(U_n(\tau))}=\lim_{n \to \infty} \dfrac{2u_n(\tau)^{-\frac{1}{2}}(1-\frac{1}{2}+1-\frac{1}{2^2}+1)}{3.2u_n(\tau)^{-\frac{1}{2}}}=\dfrac{3}{4}.
\end{split}
\end{equation*}
\end{remark}

\begin{example}\label{ex:nper-weighted}
Consider the same $f$ and $\zeta$ as in Example \ref{ex:nper-evendist}, but take $\psi$ to be
\begin{equation*}
\psi(x):=\begin{cases}
\modl{x-\zeta}^{-2}, \ x \in B_{\varepsilon_1}(\zeta)\\
9\modl{x-f(\zeta)}^{-2}, \ x \in B_{\varepsilon_2}(f(\zeta))\\
\modl{x-f^3(\zeta)}^{-2}, \ x \in B_{\varepsilon_3}(f^3(\zeta))\\
0, \ \text{otherwise}
\end{cases}
\end{equation*}
for some $\varepsilon_1,\varepsilon_2,\varepsilon_3>0$. Observe that, presented as in (\ref{eq:psi}),
\begin{equation*}
\psi(x)=\displaystyle\sum_{i=1}^{3} h_i(\modl{x-\xi_i})\mathbf{1}_{B_{\varepsilon_i}(\xi_i)}(x)
\end{equation*}
for $\xi_1=\zeta$, $\xi_2=f(\zeta)$, $\xi_3=f^3(\zeta)$, $h_1(t)=h_3(t)=t^{-2}$ and $h_2(t)=9t^{-2}$ (so that $\alpha=1/2$). In particular, $\mathcal{M}=\{\zeta,f(\zeta),f^3(\zeta)\}$ and equation (\ref{eq:alpha-hea}) holds with $a_n=100n^2$.

Again, $D_i=1$ for $i=1,2,3$ and $d=1$, but now $c_1=c_3=1$ and $c_2=9$. Thus, $p_1=p_3=\dfrac{1^{\frac{1}{2}}}{1^{\frac{1}{2}}+9^{\frac{1}{2}}+1^{\frac{1}{2}}}=\dfrac{1}{5}$ and $p_2=\dfrac{9^{\frac{1}{2}}}{1^{\frac{1}{2}}+2^{\frac{1}{2}}+1^{\frac{1}{2}}}=\dfrac{3}{5}$. We have $f'(x)=2$ for all $x \in [0,1]$, so that $(f^{-j})'(\xi_i)=\dfrac{1}{2^j}$ when $j<0$ for $i=1,2,3$.

If $i=2$, then $\lambda_{2,1}^{min}=\lambda_{2,1}^{max}=\dfrac{1}{2}$, giving $u_{2,1}^{min}=u_{2,1}^{max}=\left(\dfrac{1}{9}\right)^{\frac{1}{2}}\dfrac{1}{\left(\frac{1}{2}\right)}=\dfrac{2}{3}$, leading to $A^{(1)}=\{1\}$.

If $i=3$, then $\lambda_{3,1}^{min}=\lambda_{3,1}^{max}=\dfrac{1}{2^3}$, giving $u_{3,1}^{min}=u_{3,1}^{max}=\left(\dfrac{1}{1}\right)^{\frac{1}{2}}\dfrac{1}{\left(\frac{1}{2^3}\right)}=8$, and $\lambda_{3,2}^{min}=\lambda_{3,2}^{max}=\dfrac{1}{2^2}$, giving $u_{3,2}^{min}=u_{3,2}^{max}=\left(\dfrac{9}{1}\right)^{\frac{1}{2}}\dfrac{1}{\left(\frac{1}{2^2}\right)}=12$, leading to $A^{(2)}=\emptyset$.

Applying Theorem \ref{thm:piling-corr-nper}, we conclude that the piling process is one of the bi-infinite sequences $(Z_j)_{j \in \mathbb{Z}}$
$$\left(\dots,\infty,U,U.2\left(\dfrac{1}{2}\right)^{\frac{1}{9}},\infty,U.2^3,\infty,\dots\right)$$
$$\left(\dots,\infty,U,\infty,\dots\right)$$
with probability $\dfrac{1}{5}$, where $U$ is uniformly distributed on $[0,1]$; with probability $\dfrac{3}{5}.\dfrac{2}{3}=\dfrac{2}{5}$ the bi-infinite sequence $(Z_j)_{j \in \mathbb{Z}}$
$$\left(\dots,\infty,U_0,\infty,U_0.2^2.9^{\frac{1}{2}},\infty,\dots\right)$$
where $U_0$ is uniformly distributed on $[0,2/3]$; and with probability $\dfrac{3}{5}.\left(1-\dfrac{2}{3}\right)=\dfrac{1}{5}$
$$\left(\dots,U_1.\dfrac{1}{2}.9^{\frac{1}{2}},U_1,\infty,U_1.2^2.9^{\frac{1}{2}},\infty,\dots\right)$$
where $U_1$ is uniformly distributed on $[2/3,1]$.

Again, in all the sequences, $U$ (resp. $U_0,U_1$) is at index $j=0$ and the entries that are visibly different from $\infty$ are the only such entries.
\end{example}
\begin{remark}\label{rmk:EI-ex-2}
For Example \ref{ex:nper-weighted}, we have
\begin{equation*}
\mu(U_n(\tau))=\displaystyle\sum_{i=1}^{3} \mu(B_{h_i^{-1}(u_n(\tau))}(\xi_i))=2u_n(\tau)^{-\frac{1}{2}}+2.3u_n(\tau)^{-\frac{1}{2}}+2u_n(\tau)^{-\frac{1}{2}}=5.2u_n(\tau)^{-\frac{1}{2}}.
\end{equation*}
Let $q_n=m_3-m_1=3$. From \cite[Corollary 4.5]{16correlated},
\begin{equation*}
\begin{split}
\mu(U_n^{(q_n)}(\tau))&=\mu(B_{h_1^{-1}(u_n(\tau))}(\xi_1))-\frac{1}{2^3}\mu(B_{h_3^{-1}(u_n(\tau))}(\xi_3))\\
&\ +\mu(B_{h_2^{-1}(u_n(\tau))}(\xi_2))-\frac{1}{2^2}\mu(B_{h_3^{-1}(u_n(\tau))}(\xi_3))\\
&\ +\mu(B_{h_3^{-1}(u_n(\tau))}(\xi_3))\\
&=2u_n(\tau)^{-\frac{1}{2}}(1-\frac{1}{2^3}+3-\frac{1}{2^2}+1).
\end{split}
\end{equation*}
Thus, the extremal index is
\begin{equation*}
\begin{split}
\vartheta=\lim_{n \to \infty} \dfrac{\mu(U_n^{(q_n)}(\tau))}{\mu(U_n(\tau))}=\lim_{n \to \infty} \dfrac{2u_n(\tau)^{-\frac{1}{2}}(1-\frac{1}{2^3}+3-\frac{1}{2^2}+1)}{5.2u_n(\tau)^{-\frac{1}{2}}}=\dfrac{37}{40}.
\end{split}
\end{equation*}
\end{remark}

\begin{example}\label{ex:2-exp-nper}
Let $f(x,y)=(2x \mod 1,3y \mod 1)$, $(x,y) \in [0,1]^2$, and $\mu= $ Lebesgue measure on $[0,1]^2$ (invariant for $f$). Take $\zeta=(\zeta_x,\zeta_y)=\left(\dfrac{1}{\sqrt{2}},\dfrac{1}{\sqrt{2}}\right)$ (non-periodic), and define the observable $\Psi$ as
\begin{equation*}
\Psi(x,y):=\begin{cases}
\norm{(x,y)-\zeta}^{-4}\dfrac{(x-\zeta_x,y-\zeta_y)}{\norm{(x,y)-\zeta}}, \ (x,y) \in B_{\varepsilon_1}(\zeta)\\
256\norm{(x,y)-f(\zeta)}^{-4}\dfrac{(x-f(\zeta)_x,y-f(\zeta)_y)}{\norm{(x,y)-f(\zeta)}}, \ (x,y) \in B_{\varepsilon_2}(f(\zeta))\\
0, \ \text{otherwise}
\end{cases}
\end{equation*}
for some $\varepsilon_1,\varepsilon_2>0$, where $f(\zeta)=(f(\zeta)_x,f(\zeta)_y)$. Observe that, presented as in (\ref{eq:psi}),
\begin{equation*}
\Psi(x,y)=\displaystyle\sum_{i=1}^{2} h_i(\text{dist}((x,y),\xi_i))\dfrac{\Phi_{\xi_i}^{-1}((x,y))}{\norm{\Phi_{\xi_i}^{-1}((x,y))}}\mathbf{1}_{B_{\varepsilon_i}(\xi_i)}((x,y))
\end{equation*}
for $\xi_1=\zeta$, $\xi_2=f(\zeta)$, $h_1(t)=t^{-4}$ and $h_2(t)=256t^{-4}$ (so that $\alpha=1/4$), and $\Phi_{\xi_i}^{-1}:B_{\varepsilon_i}(\xi_i) \to B_{\varepsilon_i}(0)$ being the translation by $-\xi_i$ for $i=1,2$. In particular, $\mathcal{M}=\{\zeta,f(\zeta)\}$ and equation (\ref{eq:alpha-hea}) holds with $a_n=\dfrac{289}{256}n^2$.
	
Since $\mu=\text{Lebesgue}$, we have $D_1=D_2=1$. Also, $d=2$ and $c_1=1$ and $c_2=256$. Thus, $p_1=\dfrac{1^{\frac{1}{2}}}{1^{\frac{1}{2}}+256^{\frac{1}{2}}}=\dfrac{1}{17}$ and $p_2=\dfrac{256^{\frac{1}{2}}}{1^{\frac{1}{2}}+256^{\frac{1}{2}}}=\dfrac{16}{17}$.
	
If $i=2$, then $\lambda_{2,1}^{\min}=\dfrac{1}{3}$ and $\lambda_{2,1}^{\max}=\dfrac{1}{2}$, giving $u_{2,1}^{\min}=\left(\dfrac{1}{256}\right)^{\frac{1}{4}}\cdot\dfrac{1}{\left(\frac{1}{3}\right)}=\dfrac{3}{4}$ and\\
$u_{2,1}^{\max}=\left(\dfrac{1}{256}\right)^{\frac{1}{4}}\cdot\dfrac{1}{\left(\frac{1}{2}\right)}=\dfrac{1}{2}$, leading to $A^{(2)}=\{1\}$. Also, we have the increasing order $0 \leq u_{2,1}^{\max} \leq u_{2,1}^{\min} \leq 1$.
	
Notice that $(Df_{\zeta})^j(\theta_x,\theta_y)=(Df_{f(\zeta)})^j(\theta_x,\theta_y)=(2^j\theta_x,3^j\theta_y)$, where $j \in \mathbb{Z}$.
	
Applying Theorem \ref{thm:piling-corr-nper}, we conclude that the piling process is
\begin{enumerate}[(0)]
\item with probability $\dfrac{1}{17}$ the bi-infinite sequence $(Z_j)_{j \in \mathbb{Z}}$ with:\begin{enumerate}[(i)]
\item entry $U.(\Theta_x,\Theta_y)$ at $j=0$
\item entry $U.(2\Theta_x,3\Theta_y)\cdot\dfrac{1}{4}$ at $j=1$
\item $\infty$ for all other positive indices $j$
\item $\infty$ for all negative indices $j$
\end{enumerate}
where $U$ is uniformly distributed on $[0,1]$, $\Theta$ is uniformly distributed on $\mathbb{S}^{1}$, and $U$ and $\Theta$ are independent;
\end{enumerate}
\begin{enumerate}[(I)]
\item with probability $\dfrac{16}{17}\cdot\dfrac{1}{2}=\dfrac{8}{17}$ the bi-infinite sequence $(Z_j)_{j \in \mathbb{Z}}$ with:\begin{enumerate}[(i)]
\item entry $U_0.(\Theta_x,\Theta_y)$ at $j=0$
\item $\infty$ for all positive indices $j$
\item $\infty$ for all negative indices $j$
\end{enumerate}
where $U_0$ is uniformly distributed on $[0,1/2)$, $\Theta$ is uniformly distributed on $\mathbb{S}^{1}$, and $U_0$ and $\Theta$ are independent;
\item with probability $\dfrac{16}{17}\left(\dfrac{3}{4}-\dfrac{1}{2}\right)=\dfrac{4}{17}$ the bi-infinite sequence $(Z_j)_{j \in \mathbb{Z}}$ with:\begin{enumerate}[(i)]
\item entry $U_{1'}.(\Theta_x,\Theta_y)$ at $j=0$
\item $\infty$ for all positive indices $j$
\item entry $U_{1'}.\left(\dfrac{1}{2}\Theta_x,\dfrac{1}{3}\Theta_y\right)\cdot 4$ at $j=-1$
\item $\infty$ for all other negative indices $j$
\end{enumerate}
where $U_{1'}$ is uniformly distributed on $[1/2,3/4)$ and $\Theta \given \{U_{1'}=z\}$ is uniformly distributed on $\left\lbrace (\theta_x,\theta_y) \in \mathbb{S}^{1}:\left\lVert\left(\dfrac{1}{2}\theta_x,\dfrac{1}{3}\theta_y\right)\right\rVert \geq \dfrac{1}{4z} \right\rbrace$;
\item with probability $\dfrac{16}{17}\left(1-\dfrac{3}{4}\right)=\dfrac{4}{17}$ the bi-infinite sequence $(Z_j)_{j \in \mathbb{Z}}$ with:\begin{enumerate}[(i)]
\item entry $U_{1}.(\Theta_x,\Theta_y)$ at $j=0$
\item $\infty$ for all positive indices $j$
\item entry $U_{1}.\left(\dfrac{1}{2}\Theta_x,\dfrac{1}{3}\Theta_y\right)\cdot 4$ at $j=-1$
\item $\infty$ for all other negative indices $j$
\end{enumerate}
where $U_1$ is uniformly distributed on $[3/4,1]$, $\Theta$ is uniformly distributed on $\mathbb{S}^{1}$, and $U_1$ and $\Theta$ are independent.
\end{enumerate}
\end{example}
\begin{remark}\label{rmk:EI-ex-3}
For Example \ref{ex:2-exp-nper}, we have
\begin{equation*}
\mu(U_n(\tau))=\displaystyle\sum_{i=1}^{2} \mu(B_{h_i^{-1}(u_n(\tau))}(\xi_i))=u_n(\tau)^{-\frac{1}{2}}+16u_n(\tau)^{-\frac{1}{2}}=17u_n(\tau)^{-\frac{1}{2}}.
\end{equation*}
Let $q_n=m_2-m_1=1$. From \cite[Corollary 4.5]{16correlated},
\begin{equation*}
\mu(U_n^{(q_n)}(\tau))=\mu(B_{h_2^{-1}(u_n(\tau))}(\xi_2))=16u_n(\tau)^{-\frac{1}{2}}.
\end{equation*}
Thus, the extremal index is
\begin{equation*}
\begin{split}
\vartheta=\lim_{n \to \infty} \dfrac{\mu(U_n^{(q_n)}(\tau))}{\mu(U_n(\tau))}=\lim_{n \to \infty} \dfrac{16u_n(\tau)^{-\frac{1}{2}}}{17u_n(\tau)^{-\frac{1}{2}}}=\dfrac{16}{17}.
\end{split}
\end{equation*}
\end{remark}
\begin{remark}
Case (II) expresses that a range of thresholds $\mathcal{U} \subseteq [u_{2,1}^{max},u_{2,1}^{min})$ determines a range of unit vectors $\mathcal{T} \subseteq \mathbb{S}^{1}$ for which $u.\norm{Df_{f(\zeta)}^{-1}(\theta)}.4 \geq 1$ whenever $u \in \mathcal{U}$ and $\theta \in \mathcal{T}$. For instance, let $\mathcal{U}=[1/2,3/5)$. Then, $u=1/2$ gives us
\begin{equation*}
\norm{Df_{f(\zeta)}^{-1}(\theta)} \geq \dfrac{1}{2} \iff \left\lVert\left(\dfrac{1}{2}\theta_x,\dfrac{1}{3}\theta_y\right)\right\rVert \geq \dfrac{1}{2}.
\end{equation*}
In turn, $u=3/5$ leads to
\begin{equation*}
\norm{Df_{f(\zeta)}^{-1}(\theta)} \geq \dfrac{5}{12} \iff \left\lVert\left(\dfrac{1}{2}\theta_x,\dfrac{1}{3}\theta_y\right)\right\rVert \geq \dfrac{5}{12}.
\end{equation*}
Since $\left\lVert\left(\dfrac{1}{2}\theta_x,\dfrac{1}{3}\theta_y\right)\right\rVert =\dfrac{5}{12}$ has solution $(\theta_x,\theta_y)=\left(\dfrac{3}{2\sqrt{5}},\dfrac{\sqrt{11}}{2\sqrt{5}}\right)$ when $\theta_x,\theta_y>0$, our knowledge of the geometry of $Df_{f(\zeta)}^{-1}:T_{f(\zeta)}[0,1]^2 \to T_{\zeta}[0,1]^2$ leads to the conclusion that
\begin{equation*}
\begin{split}
\theta \in& \left[0,\tan^{-1}\left(\dfrac{\sqrt{11}}{3}\right)\right] \cup \left[\pi-\tan^{-1}\left(\dfrac{\sqrt{11}}{3}\right),\pi\right] \cup \left[\pi,\pi+\tan^{-1}\left(\dfrac{\sqrt{11}}{3}\right)\right]\\
&\cup \left[2\pi-\tan^{-1}\left(\dfrac{\sqrt{11}}{3}\right),2\pi\right].
\end{split}
\end{equation*}
\end{remark}

\subsection{Proof of the main theorem in the non-periodic case}
\begin{proof}[Proof of Theorem \ref{thm:piling-corr-nper}]
Our aim is to obtain the distribution of $(Z_j)_{j \in \mathbb{Z}}$ which is the same as the distribution of $(Y_j)_{j \in \mathbb{Z}}$ conditional on $\displaystyle\inf_{j \leq -1} \norm{Y_j} \geq 1$ (see Definition \ref{def:piling}). We take two main steps the first of which is further split into a couple of sub-steps.

\underline{\textbf{Step 1}} We check that the process $(Y_j)_{j \in \mathbb{Z}}$ is, with probability $p_i$, the bi-infinite sequence with:\begin{enumerate}[(i)]
\item entry $U.\Theta$ at $j=0$;
\item entries $U.Df_{\xi_i}^{m_l-m_i}(\Theta)\left(\dfrac{c_i}{c_l}\right)^{\alpha}$ at $j=m_l-m_i$ for all $l=i+1,\dots,N$;
\item $\infty$ for all other positive indices $j$;
\item $\infty$ for all negative indices $j$ except, possibly, $U.Df_{\xi_i}^{m_l-m_i}(\Theta)\left(\dfrac{c_i}{c_l}\right)^{\alpha}$ at $j=m_l-m_i$ for $l=1,\dots,i-1$;
\end{enumerate}
where $U$ is uniformly distributed on $[0,1]$, $\Theta$ is uniformly distributed on $\mathbb{S}^{d-1}$, and $U$ and $\Theta$ are independent.

Verifying that conditions (2)-(4) in Definition \ref{def:piling} are satisfied for $(Y_j)_{j \in \mathbb{Z}}$ as just described is straightforward, so we check that condition (1) holds.

\underline{\textbf{Sub-step 1.1}} $p_i$ is the probability that an exceedance of the threshold $u_n(\tau)$ by $\norm{\mathbf{X}_{r_n}}$ is due to a hit (at time $r_n$) to the ball around $\xi_i$ of radius $h_i^{-1}(u_n(\tau))$.

Observe that
\begin{equation*}
\{x \in \mathcal{X}: \norm{\mathbf{X}_{r_n}(x)}>u_n(\tau)\}=\left\lbrace x \in \mathcal{X}: f^{r_n}(x) \in \displaystyle\bigcup_{i=1}^{N} B_{h_i^{-1}(u_n(\tau))}(\xi_i) \right\rbrace.
\end{equation*}
We may assume that the union $\displaystyle\bigcup_{i=1}^{N} B_{h_i^{-1}(u_n(\tau))}(\xi_i)$ is disjoint as indeed it is for a sufficiently large $n$. Now, a hit to the union $\displaystyle\bigcup_{i=1}^{N} B_{h_i^{-1}(u_n(\tau))}(\xi_i)$ is indeed a hit to the ball $B_{h_i^{-1}(u_n(\tau))}(\xi_i)$, where $i \in \{1,\dots,N\}$, with probability
\begin{equation*}
p_i=\dfrac{\mu(f^{-r_n}(B_{h_i^{-1}(u_n(\tau))}(\xi_i)))}{\displaystyle\sum_{k=1}^{N}\mu(f^{-r_n}(B_{h_k^{-1}(u_n(\tau))}(\xi_k)))}=\dfrac{\mu(B_{h_i^{-1}(u_n(\tau))}(\xi_i))}{\displaystyle\sum_{k=1}^{N}\mu(B_{h_k^{-1}(u_n(\tau))}(\xi_k))} \sim \dfrac{D_i.\text{Leb}(B_{h_i^{-1}(u_n(\tau))}(\xi_i))}{\displaystyle\sum_{k=1}^{N}D_k.\text{Leb}(B_{h_k^{-1}(u_n(\tau))}(\xi_k))}
\end{equation*}
where the second equality follows by $f$-invariance of $\mu$ and the asymptotic equivalence is derived from \ref{itm:R3} in the Introduction.
In fact, we may further write
\begin{equation}\label{eq:pi}
p_i \sim \dfrac{D_i.(h_i^{-1}(u_n(\tau)))^d}{\displaystyle\sum_{k=1}^{N}D_k.(h_k^{-1}(u_n(\tau)))^d}=\dfrac{D_i\left(\dfrac{u_n(\tau)}{c_i}\right)^{-\alpha d}}{\displaystyle\sum_{k=1}^{N}D_k\left(\dfrac{u_n(\tau)}{c_k}\right)^{-\alpha d}}=\dfrac{D_ic_i^{\alpha d}}{\displaystyle\sum_{k=1}^{N}D_kc_k^{\alpha d}}.
\end{equation}

\underline{\textbf{Sub-step 1.2}} Assume a hit at time $r_n$ to the ball around $\xi_i$ of radius $h_i^{-1}(u_n(\tau))$. Then, $(Y_j)_{j \in \mathbb{Z}}$ is as described by (i)-(iv) in Step 1.

Since $\Psi(x)=\displaystyle\sum_{i=1}^{N}h_i(\text{dist}(x,\xi_i))\dfrac{\Phi^{-1}_{\xi_i}(x)}{\norm{\Phi^{-1}_{\xi_i}(x)}}\mathbf{1}_{W_i}(x)$ and $f^{r_n}(x) \in W_i$,
\begin{equation*}
\dfrac{\mathbf{X}_{r_n}(x)}{\norm{\mathbf{X}_{r_n}(x)}}=\dfrac{\Psi(f^{r_n}(x))}{\norm{\Psi(f^{r_n}(x))}}=\dfrac{\Phi^{-1}_{\xi_i}(f^{r_n}(x))}{\norm{\Phi^{-1}_{\xi_i}(f^{r_n}(x))}}=w.
\end{equation*}
Dropping the dependence on $x$, we have that $\dfrac{\mathbf{X}_{r_n}}{\norm{\mathbf{X}_{r_n}}}=\Theta$. As, in a sufficiently small neighbourhood of $\mathcal{M}$, $\mu$ looks like the Lebesgue measure (recall \ref{itm:R3} in the Introduction) it follows that $\Theta$ is uniformly distributed on $\mathbb{S}^{d-1}$. As $Y_0 \sim \dfrac{u_n^{-1}(\norm{\mathbf{X}_{r_n}})}{\tau}\dfrac{\mathbf{X}_{r_n}}{\norm{\mathbf{X}_{r_n}}}$ and, from \cite[Lemma 3.9]{20enriched}, $\norm{Y_0}$ is uniformly distributed on $[0,1]$, we must have $\dfrac{u_n^{-1}(\norm{\mathbf{X}_{r_n}})}{\tau} \sim U$ where $U$ is uniformly distributed on $[0,1]$. Thus,
\begin{equation*}
Y_0 \sim U.\Theta
\end{equation*}
where $U$ is uniformly distributed on $[0,1]$, $\Theta$ is uniformly distributed on $\mathbb{S}^{d-1}$ and $U$ and $\Theta$ are independent. Therefore, (i) for $(Y_j)_{j \in \mathbb{Z}}$ is satisfied.

Observe that continuity of $f$ at $\xi_i$ guarantees that $f^{r_n+j}(x)$ belongs to a small neighbourhood of $f^{j}(\xi_i)$ as long as $f^{r_n}(x)$ belongs to a sufficiently small neighbourhood of $\xi_i$. We have
\begin{equation}\label{eq:lin-approx}
\text{dist}(f^{r_n+j}(x),f^{j}(\xi_i)) \sim \norm{Df_{\xi_i}^{j}(w)}.\text{dist}(f^{r_n}(x),\xi_i)
\end{equation}
where $w=\dfrac{\Phi_{\xi_i}^{-1}(f^{r_n}(x))}{\norm{\Phi_{\xi_i}^{-1}(f^{r_n}(x))}}$.

Now, since the sequence $\mathbf{X}_0,\mathbf{X}_1,\dots$ has an $\alpha$-regularly varying tail we use the formula $u_n^{-1}(z)=\left(\dfrac{z}{a_n}\right)^{-\alpha}$.

Notice that $f^{m_l-m_i}(\xi_i)=\xi_l$. Thus, $f^{r_n+m_l-m_i}(x)$ belongs to a small neighbourhood of $\xi_l$ and we may write
\begin{equation*}
\norm{\mathbf{X}_{r_n+m_l-m_i}(x)}=h_l(\text{dist}(f^{r_n+m_l-m_i}(x),\xi_l))=h_l(\text{dist}(f^{r_n+m_l-m_i}(x),f^{m_l-m_i}(\xi_i))).
\end{equation*}
So, for all $l=i+1,\dots,N$,
\begin{equation*}
\begin{split}
u_n^{-1}(\norm{\mathbf{X}_{r_n+m_l-m_i}(x)})&=u_n^{-1}(h_l(\text{dist}(f^{r_n+m_l-m_i}(x),f^{m_l-m_i}(\xi_i))))\\
&=\left(\dfrac{h_l(\text{dist}(f^{r_n+m_l-m_i}(x),f^{m_l-m_i}(\xi_i)))}{a_n}\right)^{-\alpha}
\end{split}
\end{equation*}
which, by (\ref{eq:lin-approx}), can be rewritten as
\begin{equation}\label{eq:j-positive-1}
\begin{split}
u_n^{-1}(\norm{\mathbf{X}_{r_n+m_l-m_i}(x)}) &\sim \left(\dfrac{h_l(\norm{Df_{\xi_i}^{m_l-m_i}(w)}.\text{dist}(f^{r_n}(x),\xi_i))}{a_n}\right)^{-\alpha}\\
&=\dfrac{c_l^{-\alpha}\norm{Df_{\xi_i}^{m_l-m_i}(w)}.\text{dist}(f^{r_n}(x),\xi_i)}{a_n^{-\alpha}}.
\end{split}
\end{equation}

In fact, $f^{r_n}(x) \in B_{h_i^{-1}(u_n(\tau))}(\xi_i)$ corresponds to
\begin{equation*}
h_i(\text{dist}(f^{r_n}(x),\xi_i))>u_n(\tau)
\end{equation*}
which together with (\ref{eq:gen-u_n^{-1}}) implies
\begin{equation*}
\tau \geq u_n^{-1}(h_i(\text{dist}(f^{r_n}(x),\xi_i))) \iff \tau=\dfrac{u_n^{-1}(h_i(\text{dist}(f^{r_n}(x),\xi_i)))}{v}
\end{equation*}
where $v \in [0,1]$. Using again the formula for $u_n^{-1}$, it follows
\begin{equation}\label{eq:j-positive-2}
\tau=\dfrac{1}{v}\left(\dfrac{h_i(\text{dist}(f^{r_n}(x),\xi_i))}{a_n}\right)^{-\alpha}=\dfrac{1}{v}\dfrac{c_i^{-\alpha}\text{dist}(f^{r_n}(x),\xi_i)}{a_n^{-\alpha}}.
\end{equation}
Since $\Psi(x)=\displaystyle\sum_{i=1}^{N}h_i(\text{dist}(x,\xi_i))\dfrac{\Phi^{-1}_{\xi_i}(x)}{\norm{\Phi^{-1}_{\xi_i}(x)}}\mathbf{1}_{W_i}(x)$, $f^{r_n}(x) \in W_i$  and $f^{r_n+m_l-m_i}(x) \in W_l$,
\begin{equation}\label{eq:j-positive-3}
\begin{split}
\dfrac{\mathbf{X}_{r_n+m_l-m_i}(x)}{\norm{\mathbf{X}_{r_n+m_l-m_i}(x)}}=\dfrac{\Psi(f^{r_n+m_l-m_i}(x))}{\norm{\Psi(f^{r_n+m_l-m_i}(x))}}&=\dfrac{\Phi^{-1}_{\xi_l}(f^{r_n+m_l-m_i}(x))}{\norm{\Phi^{-1}_{\xi_l}(f^{r_n+m_l-m_i}(x))}}\\
&=\dfrac{Df_{\xi_i}^{m_l-m_i}(\Phi_{\xi_i}^{-1}(f^{r_n}(x))}{\norm{Df_{\xi_i}^{m_l-m_i}(\Phi_{\xi_i}^{-1}(f^{r_n}(x))}}\\
&=\dfrac{Df_{\xi_i}^{m_l-m_i}(w)}{\norm{Df_{\xi_i}^{m_l-m_i}(w)}}.
\end{split}
\end{equation}
Putting together (\ref{eq:j-positive-1}), (\ref{eq:j-positive-2}) and (\ref{eq:j-positive-3}), we conclude that, for all $l=i+1,\dots,N$,
\begin{equation*}
\dfrac{u_n^{-1}(\norm{\mathbf{X}_{r_n+m_l-m_i}(x)})}{\tau}\dfrac{\mathbf{X}_{r_n+m_l-m_i}(x)}{\norm{\mathbf{X}_{r_n+m_l-m_i}(x)}} \sim v\left(\dfrac{c_i}{c_l}\right)^{\alpha}Df_{\xi_i}^{m_l-m_i}(w)
\end{equation*}
where $v \in [0,1]$ and $w \in \mathbb{S}^{d-1}$.

Finally, $v$ and $w$ are attached to a particular observation, labelling with a magnitude and a direction the hit at time $r_n$ to the ball around $\xi_i$ of radius $h_i^{-1}(u_n(\tau))$. Dropping the dependence on $x$, we have
\begin{equation*}
\dfrac{u_n^{-1}(\norm{\mathbf{X}_{r_n+m_l-m_i}})}{\tau}\dfrac{\mathbf{X}_{r_n+m_l-m_i}}{\norm{\mathbf{X}_{r_n+m_l-m_i}}} \sim U.Df_{\xi_i}^{m_l-m_i}(\Theta)\left(\dfrac{c_i}{c_l}\right)^{\alpha}
\end{equation*}
where $U$ is uniformly distributed on $[0,1]$, $\Theta$ is uniformly distributed on $\mathbb{S}^{d-1}$, and $U$ and $\Theta$ are independent. So, (ii) for $(Y_j)_{j \in \mathbb{Z}}$ is satisfied.

For all positive indices $j \notin \{m_l-m_i: l=i+1,\dots,N\}$, we claim that
\begin{equation*}
\lim_{n \to \infty}\dfrac{u_n^{-1}(\norm{\mathbf{X}_{r_n+j}})}{\tau}\dfrac{\mathbf{X}_{r_n+j}}{\norm{\mathbf{X}_{r_n+j}}}=\infty.
\end{equation*}
Suppose otherwise, so that $\displaystyle\lim_{n \to \infty} \dfrac{u_n^{-1}(\norm{\mathbf{X}_{r_n+j}})}{\tau}=c \in \mathbb{R}$ for some positive index $j \notin \{m_l-m_i: l=i+1,\dots,N\}$. Then, at time $r_n+j$ there is a hit to $B_{u_n(\tau')}(\xi_k)$ for some $\tau' \in (0,+\infty)$ and $k \in \{1,\dots,N\}$. As $u_n(\tau) \underset{\scriptsize{n \to \infty}}{\longrightarrow}+\infty$, for any $\tau \in (0,+\infty)$, then $f^{r_n}(x) \underset{\scriptsize{n \to \infty}}{\longrightarrow} \xi_i$ and $f^{r_n+j}(x) \underset{\scriptsize{n \to \infty}}{\longrightarrow} \xi_k$. By continuity of $f$ at $\xi_i$, $f^{r_n+j}(x) \underset{\scriptsize{n \to \infty}}{\longrightarrow} f^j(\xi_i)$ and so $f^j(\xi_i)=\xi_k$. However, by definition, $\xi_k=f^{m_k-m_i}(\xi_i)$ where $k \in \{1,\dots,N\}$. It follows that $\xi_i$ is periodic, thus $\zeta$ is periodic contrary to our assumption. So, (iii) for $(Y_j)_{j \in \mathbb{Z}}$ is satisfied.

We are left to check (iv): $Y_j$ is equal to $\infty$ except, possibly, for a finite number of negative indices $j=m_1-m_i,\dots,m_{i-1}-m_i$. In fact, if the visit to a neighbourhood of $\xi_i$ (at time $r_n$) is preceded by a visit to a neighbourhood of $\xi_{i-1}$ (at time $r_n-(m_i-m_{i-1})$) then, arguing as in the justification of (ii) above,
\begin{equation*}
\dfrac{u_n^{-1}(\norm{\mathbf{X}_{r_n+m_{i-1}-m_i}})}{\tau}\dfrac{\mathbf{X}_{r_n+m_{i-1}-m_i}}{\norm{\mathbf{X}_{r_n+m_{i-1}-m_i}}} \sim U.Df_{\xi_i}^{m_{i-1}-m_i}(\Theta)\left(\dfrac{c_i}{c_{i-1}}\right)^{\alpha}.
\end{equation*}
On the other hand, if the visit to a neighbourhood of $\xi_i$ (at time $r_n$) is preceded by a visit to a neighbourhood of $f^{-(m_i-m_{i-1})}(\xi_i) \setminus \{\xi_{i-1}\}$ (at time $r_n-(m_i-m_{i-1})$) then
\begin{equation*}
\lim_{n \to \infty}\dfrac{u_n^{-1}(\norm{\mathbf{X}_{r_n+m_{i-1}-m_i}})}{\tau}\dfrac{\mathbf{X}_{r_n+m_{i-1}-m_i}}{\norm{\mathbf{X}_{r_n+m_{i-1}-m_i}}}=\infty
\end{equation*}
by definition of the observable $\Psi$ (\textit{i.e.} in a neighbourhood of a $(m_i-m_{i-1})$-th pre-image of $\xi_i$ which is not $\xi_{i-1}$ then $\Psi \equiv 0$). Thus, inductively, we see that among the indices $j=m_1-m_i,\dots,m_{i-1}-m_i$ we can have $\displaystyle\lim_{n \to \infty}\dfrac{u_n^{-1}(\norm{\mathbf{X}_{r_n+j}})}{\tau}\dfrac{\mathbf{X}_{r_n+j}}{\norm{\mathbf{X}_{r_n+j}}}$ different or equal to $\infty$. By analogous reasoning to what was used to claim (iii) just above, we see that the negatively indexed entries corresponding to indices $j \neq m_1-m_i,\dots,m_{i-1}-m_i$ must all be $\infty$.

\underline{\textbf{Step 2}} The distribution of $(Z_j)_{j \in \mathbb{Z}}$ is given by (0)-(IV).

We look at the norms of the negatively indexed entries in $(Y_j)_{j \in \mathbb{Z}}$, more specifically for the entries corresponding to $j=m_l-m_i$ for all $l=1,\dots,i-1$ (as $Y_j=\infty$ for all the other negative indices $j$ as determined by (iv) we have just shown).

First, assume $\lambda_{i,l}^{max}\left(\dfrac{c_i}{c_l}\right)^{\alpha}<1$ for all $l=1,\dots,i-1$. Then, for any $u \in [0,1]$,\\
$u.\norm{Df_{\xi_i}^{m_l-m_i}(\Theta)}\left(\dfrac{c_i}{c_l}\right)^{\alpha}<1$ (a.s.) for all $l=1,\dots,i-1$.
Since $\lambda_{i,l}^{max}\left(\dfrac{c_i}{c_l}\right)^{\alpha}<1$ is equivalent to $u_{i,l}^{max}>1$ (which implies that $u_{i,l}^{min}>1$) then $A^{(i)}=\emptyset$ and $\norm{Y_j} \geq 1$ when $j=m_l-m_i$ and $l=1,\dots,i-1$ can only be so if $Y_j=\infty$. We have dealt with case (0) in the statement of the Proposition.

On the other hand, assume that $\lambda_{i,l}^{min}\left(\dfrac{c_i}{c_l}\right)^{\alpha} \geq 1$ for some $l=1,\dots,i-1$ that we now fix. Then, for some $u \in [0,1]$, $u.\norm{Df_{\xi_i}^{m_l-m_i}(\Theta)}\left(\dfrac{c_i}{c_l}\right)^{\alpha} \geq 1$ (a.s.). In particular, $u_{i,l}^{min} \leq 1$, as in the case where $A^{(i)} \neq \emptyset$. We consider $\#A^{(i)}=2$ as the general case follows analogously. So, we have $u_{i,l_1}^{max} \leq u_{i,l_1}^{min} \leq u_{i,l_2}^{max} \leq u_{i,l_2}^{min} \leq 1$ for $l_1,l_2 \in \{1,\dots,i-1\}$. If $u \in [0,1]$ is such that $u < u_{i,l_1}^{max}$ it follows that $u.\norm{Df_{\xi_i}^{m_l-m_i}(\Theta)}\left(\dfrac{c_i}{c_l}\right)^{\alpha}<1$ (a.s.) at $j=m_l-m_i$ for both $l=l_1$ and $l=l_2$  and, therefore, the entries corresponding to indices $j=m_l-m_i$ for both $l=l_1$ and $l=l_2$ must be $\infty$; since $u < u_{i,l_1}^{max}$ has probability $\mathbb{P}(U < u_{i,l_1}^{max})=u_{i,l_1}^{max}$ ($U$ is uniformly distributed on $[0,1]$), we are in case (I).

If $u \in [0,1]$ is such that $u_{i,l_1}^{min} \leq u < u_{i,l_2}^{max}$ it follows that $u.\norm{Df_{\xi_i}^{m_l-m_i}(\Theta)}\left(\dfrac{c_i}{c_l}\right)^{\alpha} \geq 1$ (a.s.) at $j=m_l-m_i$ for $l=l_1$ but not for $l=l_2$, so that only the entry corresponding to index $j=m_l-m_i$ for $l=l_1$ can be different from $\infty$ (and indeed equal to $u.Df_{\xi_i}^{m_l-m_i}(\Theta)\left(\dfrac{c_i}{c_l}\right)^{\alpha}$ where $u$ is chosen equally likely among the elements of $[u_{i,l_1}^{min},u_{i,l_2}^{max})$); since $u_{i,l_1}^{min} \leq u < u_{i,l_2}^{max}$ has probability $\mathbb{P}(u_{i,l_1}^{min} \leq U < u_{i,l_2}^{max})=u_{i,l_2}^{max}-u_{i,l_1}^{min}$, we are in case (III).

If $u \in [0,1]$ is such that $u \geq u_{i,l_2}^{min}$, then $u.\norm{Df_{\xi_i}^{m_l-m_i}(\Theta)}\left(\dfrac{c_i}{c_l}\right)^{\alpha} \geq 1$ at $j=m_l-m_i$ for both $l=l_1$ and $l=l_2$ leading to the entries corresponding to indices $j=m_l-m_i$ for $l=l_1$ and $l=l_2$ both being different from $\infty$ (and indeed equal to $u.Df_{\xi_i}^{m_l-m_i}(\Theta)\left(\dfrac{c_i}{c_l}\right)^{\alpha}$ where $u$ is chosen equally likely among the elements of $[u_{i,l_2}^{min},1)$); since $u \geq u_{i,l_2}^{min}$ has probability $\mathbb{P}(U \geq u_{i,l_2}^{min})=1-u_{i,l_2}^{min}$, we are in case (IV).

Finally, if $u \in [0,1]$ is such that $u_{i,l_1}^{max} \leq u < u_{i,l_1}^{min}$, which occurs with probability $u_{i,l_1}^{min}-u_{i,l_1}^{max}$, then $u.\norm{Df_{\xi_i}^{m_l-m_i}(w)}\left(\dfrac{c_i}{c_l}\right)^{\alpha} \geq 1$ for all $w \in \mathbb{S}^{d-1}$ such that $\norm{Df_{\xi_i}^{m_l-m_i}(w)} \geq \dfrac{1}{u}\left(\dfrac{c_l}{c_i}\right)^{\alpha}$; thus, $\Theta \given \{U_{1'}=u\}$ is uniformly distributed on\\ $\left\lbrace w \in \mathbb{S}^{d-1}:\norm{Df_{\xi_i}^{m_l-m_i}(w)} \geq \dfrac{1}{u}\left(\dfrac{c_l}{c_i}\right)^{\alpha} \right\rbrace$, for $U_{1'}$ uniformly distributed on $[u_{i,l_1}^{max},u_{i,l_1}^{min})$, and we are in case (II) (the other situation where $u_{i,l_2}^{max} \leq u < u_{i,l_2}^{min}$ is entirely analogous).
\end{proof}
The following corollary justifies why we imposed from the beginning that all the $h_i$ in $\Psi$ have the same index $\alpha$.
\begin{corollary}\label{cor:alphas}
Let $f$, $\mu$ and $\Psi$ be as specified in the Introduction with $\mathcal{M}$ as in \ref{itm:A}. If the $h_i$ (as in (\ref{eq:hi})) were allowed not all with the same $\alpha$ then $\mathbb{P}(S=\xi_i)=p_i$ would define a degenerate random variable $S$.
\end{corollary}
\begin{proof}
Let $\mathcal{M}=\{\xi_1,\xi_2\}=\{\zeta,f(\zeta)\}$, with $h_1(x)=c_1x^{-\frac{1}{\alpha_1}}$ and $h_2(x)=c_2x^{-\frac{1}{\alpha_2}}$ where $\alpha_1<\alpha_2$. Assume $d=1$. Then, making use of (\ref{eq:pi}),
\begin{equation*}
\mathbb{P}(S=\xi_1)=p_1 \sim \dfrac{D_1\left(\dfrac{u_n(\tau)}{c_1}\right)^{-\alpha_1}}{D_1\left(\dfrac{u_n(\tau)}{c_1}\right)^{-\alpha_1}+D_2\left(\dfrac{u_n(\tau)}{c_2}\right)^{-\alpha_2}} \sim 1;
\end{equation*}
\begin{equation*}
\mathbb{P}(S=\xi_2)=p_2 \sim \dfrac{D_2\left(\dfrac{u_n(\tau)}{c_2}\right)^{-\alpha_2}}{D_1\left(\dfrac{u_n(\tau)}{c_1}\right)^{-\alpha_1}+D_2\left(\dfrac{u_n(\tau)}{c_2}\right)^{-\alpha_2}} \sim 0.
\end{equation*}
The general case follows analogously.
\end{proof}

\subsection{Periodic case}
\begin{theorem}\label{thm:piling-corr-per}
Let $f$, $\mu$ and $\Psi$ be as specified in Section \ref{sec:intro} with $\mathcal{M}$ as in \ref{itm:A}, where $\zeta$ is a periodic point of prime period $q$. Assume that $(\mathbf{X}_n)_{n \in \mathbb{N}_0}$ has an $\alpha$-regularly varying tail, where $\alpha \in (0,1)$. Additionally, assume that $f$ is uniformly expanding along the orbit of $\zeta$. For all $i \in \mathcal{I}$, define $p_i=\dfrac{D_ic_i^{\alpha d}}{\sum_{k=1}^{N}D_kc_k^{\alpha d}}$. Let $A^{(i)}$ be as defined in \ref{itm:A2}. If $A^{(i)}=\emptyset$, the piling process is 
\begin{enumerate}[(0)]
\item\label{itm:(0)-piling-fin-per} with probability $p_i$ the bi-infinite sequence $(Z_j)_{j \in \mathbb{Z}}$ with:
\begin{enumerate}[(i)]
\item entry $U.\Theta$ at $j=0$;
\item entries $U.Df_{\xi_i}^{j}(\Theta)\left(\dfrac{c_i}{c_l}\right)^{\alpha}$ at $j=m_l-m_i+qs$ for all $l \in \mathcal{I}$ and $s \in \mathbb{N}_0$ such that $m_l-m_i+qs>0$;
\item $\infty$ for all other positive indices $j$;
\item $\infty$ for all negative indices $j$;
\end{enumerate}
where $U$ is uniformly distributed on $[0,1]$, $\Theta$ is uniformly distributed on $\mathbb{S}^{d-1}$, and $U$ and $\Theta$ are independent.
\end{enumerate}
If $A^{(i)} \neq \emptyset$, assume there exists an increasing order
\begin{equation*}
0 \leq u_{i,l_1}^{max} \leq u_{i,l_1}^{min} \leq u_{i,l_2}^{max} \leq u_{i,l_2}^{min} \leq \dots \leq u_{i,l_{\#A^{(i)}}}^{max} \leq u_{i,l_{\#A^{(i)}}}^{min} \leq 1
\end{equation*}
and for $u_{i,l_p}^{min}=u_{i,l,s}^{min}$ (resp. $u_{i,l_p}^{max}=u_{i,l,s}^{max}$), $p \in \{1,\dots,\#A^{(i)}\}$, let $\rho(u_{i,l_p}^{min}):=m_l-m_i-qs$ and $\rho(u_{i,l_p}^{max}):=m_l-m_i-qs$, so that we abbreviate to $\rho(u_{i,l_p}):=m_l-m_i-qs$. Then, the piling process is
\begin{enumerate}[(I)]
\item with probability $p_iu_{i,l_1}^{max}$ the bi-infinite sequence $(Z_j)_{j \in \mathbb{Z}}$ with:\begin{enumerate}[(i)]
\item entry $U_0.\Theta$ at $j=0$;
\item entries $U_0.Df_{\xi_i}^{j}(\Theta)\left(\dfrac{c_i}{c_l}\right)^{\alpha}$ at $j=m_l-m_i+qs$ for all $l \in \mathcal{I}$ and $s \in \mathbb{N}_0$ such that $m_l-m_i+qs>0$;
\item $\infty$ for all other positive indices $j$;
\item $\infty$ for all negative indices $j$;
\end{enumerate}
where $U_0$ is uniformly distributed on $[0,u_{i,l_1}^{max})$, $\Theta$ is uniformly distributed on $\mathbb{S}^{d-1}$, and $U_0$ and $\Theta$ are independent;
\item with probability $p_i(u_{i,l_p}^{min}-u_{i,l_p}^{max})$, where $p \in \{1,\dots,\#A^{(i)}\}$, the bi-infinite sequence $(Z_j)_{j \in \mathbb{Z}}$ with:\begin{enumerate}[(i)]
\item entry $U_{p'}.\Theta$ at $j=0$;
\item entries $U_{p'}.Df_{\xi_i}^{j}(\Theta)\left(\dfrac{c_i}{c_l}\right)^{\alpha}$ at $j=m_l-m_i+qs$ for all $l \in \mathcal{I}$ and $s \in \mathbb{N}_0$ such that $m_l-m_i+qs>0$;
\item $\infty$ for all other positive indices $j$;
\item entries $U_{p'}.Df_{\xi_i}^{j}(\Theta)\left(\dfrac{c_i}{c_l}\right)^{\alpha}$ at $j=\rho(u_{i,l_1}),\dots,\rho(u_{i,l_p})$;
\item $\infty$ for all other negative indices $j$;
\end{enumerate}
where $U_{p'}$ is uniformly distributed on $[u_{i,l_p}^{max},u_{i,l_p}^{min})$ and $\Theta \given \{U_{p'}=u\}$ is uniformly distributed on $\left\lbrace w \in \mathbb{S}^{d-1}:\norm{Df_{\xi_i}^{\rho(u_{i,l_p})}(w)} \geq \dfrac{1}{u}\left(\dfrac{c_{l_p}}{c_i}\right)^{\alpha} \right\rbrace$;
\item with probability $p_i(u_{i,l_{p+1}}^{max}-u_{i,l_p}^{min})$, where $p \in \{1,\dots,\#A^{(i)}-1\}$, the bi-infinite sequence $(Z_j)_{j \in \mathbb{Z}}$ with:\begin{enumerate}[(i)]
\item entry $U_p.\Theta$ at $j=0$;
\item entries $U_p.Df_{\xi_i}^{j}(\Theta)\left(\dfrac{c_i}{c_l}\right)^{\alpha}$ at $j=m_l-m_i+qs$ for all $l \in \mathcal{I}$ and $s \in \mathbb{N}_0$ such that $m_l-m_i+qs>0$;
\item $\infty$ for all other positive indices $j$;
\item entries $U_p.Df_{\xi_i}^{j}(\Theta)\left(\dfrac{c_i}{c_l}\right)^{\alpha}$ at $j=\rho(u_{i,l_1}),\dots,\rho(u_{i,l_p})$;
\item $\infty$ for all other negative indices $j$;
\end{enumerate}
where $U_p$ is uniformly distributed on $[u_{i,l_p}^{min},u_{i,l_{p+1}}^{max})$, $\Theta$ is uniformly distributed on $\mathbb{S}^{d-1}$, and $U_p$ and $\Theta$ are independent;
\item with probability $p_i.(1-u_{i,l_{\#A^{(i)}}}^{min})$ the bi-infinite sequence $(Z_j)_{j \in \mathbb{Z}}$ with:\begin{enumerate}[(i)]
\item entry $U_{\#A^{(i)}}.\Theta$ at $j=0$;
\item entries $U_{\#A^{(i)}}.Df_{\xi_i}^{j}(\Theta)\left(\dfrac{c_i}{c_l}\right)^{\alpha}$ at $j=m_l-m_i+qs$ for all $l \in \mathcal{I}$ and $s \in \mathbb{N}_0$ such that $m_l-m_i+qs>0$;
\item $\infty$ for all other positive indices $j$;
\item entries $U_{\#A^{(i)}}.Df_{\xi_i}^{j}(\Theta)\left(\dfrac{c_i}{c_l}\right)^{\alpha}$ at $j$ for all $j \in A^{(i)}$;
\item $\infty$ for all other negative indices $j$;
\end{enumerate}
where $U_{\#A^{(i)}}$ is uniformly distributed on $[u_{i,l_{\#A^{(i)}}}^{min},1]$, $\Theta$ is uniformly distributed on $\mathbb{S}^{d-1}$, and $U_{\#A^{(i)}}$ and $\Theta$ are independent.
\end{enumerate}
\end{theorem}
\begin{proof}
Recall that $(Z_j)_{j \in \mathbb{Z}}$ has the distribution of $(Y_j)_{j \in \mathbb{Z}}$ conditional on $\displaystyle\inf_{j \leq -1} \norm{Y_j} \geq 1$ (see Definition \ref{def:piling}). The proof follows analogously to the proof of Theorem \ref{thm:piling-corr-nper} so we write down the same steps highlighting the differences that arise from the fact that $\zeta$ is now periodic (of prime period $q$).

\underline{\textbf{Step 1}} We check that the process $(Y_j)_{j \in \mathbb{Z}}$ is, with probability $p_i$, the bi-infinite sequence with:
\begin{enumerate}[(i)]
\item entry $U.\Theta$ at $j=0$;
\item entries $U.Df_{\xi_i}^{j}(\Theta)\left(\dfrac{c_i}{c_l}\right)^{\alpha}$ at $j=m_l-m_i+qs$, for all $l \in \mathcal{I}$ and $s \in \mathbb{N}_0$ such that $m_l-m_i+qs>0$;
\item $\infty$ for all other positive indices $j$;
\item $\infty$ for all negative indices $j$ except, possibly, $U.Df_{\xi_i}^{j}(\Theta)\left(\dfrac{c_i}{c_l}\right)^{\alpha}$ at $j=m_l-m_i-qs$, for $l \in \mathcal{I}$ and $s \in \mathbb{N}_0$ such that $m_l-m_i-qs<0$;
\end{enumerate}
where $U$ is uniformly distributed on $[0,1]$, $\Theta$ is uniformly distributed on $\mathbb{S}^{d-1}$, and $U$ and $\Theta$ are independent.

Verifying that condition (2) in Definition \ref{def:piling} is satisfied for $(Y_j)_{j \in \mathbb{Z}}$ as just described is straightforward. The (a.s.) positive exponential growth of $\norm{Df_{\xi_i}^j(\Theta)}$ for positive $j$ as well as the (a.s.) negative exponential growth of $\norm{Df_{\xi_i}^{j}(\Theta)}$ for negative $j$ lead to (3) being satisfied. Again, the (a.s.) negative exponential growth of $\norm{Df_{\xi_i}^{j}(\Theta)}$ for negative $j$ results in $\norm{Y_j} \leq 1$ for all but a finite number of negative indices $j$ so that (4) is satisfied.

We are left to check that condition (1) in Definition \ref{def:piling} holds.

\underline{\textbf{Sub-step 1.1}} $p_i$ is the probability that an exceedance of the threshold $u_n(\tau)$ by $\norm{\mathbf{X}_{r_n}}$ is due to a hit (at time $r_n$) to the ball around $\xi_i$ of radius $h_i^{-1}(u_n(\tau))$.

From the exact same reasoning as in Sub-step 1.1 in the proof of Theorem \ref{thm:piling-corr-nper}: what matters is how much a neighbourhood around $\xi_i$ (corresponding to the exceedance of a threshold $u_n(\tau)$) weighs in a neighbourhood of $\mathcal{M}$ (corresponding to the exceedance of the same $u_n(\tau)$), which is not dependent on $\zeta$ being or not periodic.

\underline{\textbf{Sub-step 1.2}} Assume a hit, at time $r_n$, to the ball around $\xi_i$ of radius $h_i^{-1}(u_n(\tau))$. Then, $(Y_j)_{j \in \mathbb{Z}}$ is as described by (i)-(iv) in Step 1.

In fact, if $\zeta$ is periodic of prime period $q$ and a hit to a neighbourhood of a certain $\xi_i$, corresponding to the exceedance of a threshold $u_n(\tau)$, occurs at time $r_n$, then hits to neighbourhoods of the same $\xi_i$ will occur at times $r_n+qs$ and at times $r_n-qs$, for all $s \in \mathbb{N}_0$. The rest follows analogously to Sub-step 1.2 in the proof of Theorem \ref{thm:piling-corr-nper} - we point out that in (iii) the contradiction is with the minimality of $q$.

\underline{\textbf{Step 2}} The distribution of $(Z_j)_{j \in \mathbb{Z}}$ is given by (0)-(IV).

This is analogous to Theorem \ref{thm:piling-corr-nper}, accounting for the changes in $(Y_j)_{j \in \mathbb{Z}}$ as discussed in Step 1.
\end{proof}

\begin{example}\label{ex:2-exp}
Let $f(x,y)=(2x \mod 1,3y \mod 1)$, $(x,y) \in [0,1]^2$, and $\mu= $ Lebesgue measure on $[0,1]^2$ (invariant for $f$). Take $\zeta=(\zeta_x,\zeta_y)=\left(\dfrac{1}{7},0\right)$ (periodic of period 3), and define the observable $\Psi$ as
\begin{equation*}
\Psi(x,y):=\begin{cases}
\norm{(x,y)-\zeta}^{-4}\dfrac{(x-\zeta_x,y-\zeta_y)}{\norm{(x,y)-\zeta}}, \ (x,y) \in B_{\varepsilon_1}(\zeta)\\
256\norm{(x,y)-f(\zeta)}^{-4}\dfrac{(x-f(\zeta)_x,y-f(\zeta)_y)}{\norm{(x,y)-f(\zeta)}}, \ (x,y) \in B_{\varepsilon_2}(f(\zeta))\\
0, \ \text{otherwise}
\end{cases}
\end{equation*}
for some $\varepsilon_1,\varepsilon_2>0$, where $f(\zeta)=(f(\zeta)_x,f(\zeta)_y)$. Observe that, presented as in (\ref{eq:psi}),
\begin{equation*}
\Psi(x,y)=\displaystyle\sum_{i=1}^{2} h_i(\text{dist}((x,y),\xi_i))\dfrac{\Phi_{\xi_i}^{-1}((x,y))}{\norm{\Phi_{\xi_i}^{-1}((x,y))}}\mathbf{1}_{B_{\varepsilon_i}(\xi_i)}((x,y))
\end{equation*}
for $\xi_1=\zeta$, $\xi_2=f(\zeta)$, $h_1(t)=t^{-4}$ and $h_2(t)=256t^{-4}$ (so that $\alpha=1/4$), and $\Phi_{\xi_i}^{-1}:B_{\varepsilon_i}(\xi_i) \to B_{\varepsilon_i}(0)$ being the translation by $-\xi_i$ for $i=1,2$. In particular, $\mathcal{M}=\{\zeta,f(\zeta)\}$ and equation (\ref{eq:alpha-hea}) holds with $a_n=\dfrac{289}{256}n^2$.

Since $\mu=\text{Lebesgue}$, we have $D_1=D_2=1$. Also, $d=2$ and $c_1=1$ and $c_2=256$. Thus, $p_1=\dfrac{1^{\frac{1}{2}}}{1^{\frac{1}{2}}+256^{\frac{1}{2}}}=\dfrac{1}{17}$ and $p_2=\dfrac{256^{\frac{1}{2}}}{1^{\frac{1}{2}}+256^{\frac{1}{2}}}=\dfrac{16}{17}$.

If $i=1$ then, for all $s \in \mathbb{N}$, $\lambda_{1,2,s}^{\min}=\dfrac{1}{3^{2s}}$, $\lambda_{1,2,s}^{\max}=\dfrac{1}{2^{2s}}$ and $\lambda_{1,1,s}^{\min}=\dfrac{1}{3^{3s}}$ $\lambda_{1,1,s}^{\max}=\dfrac{1}{2^{3s}}$, giving $u_{1,2,s}^{\min}=\left(\dfrac{256}{1}\right)^{\frac{1}{4}}\dfrac{1}{\left(\frac{1}{3^{2s}}\right)}=4.3^{2s}$, $u_{1,2,s}^{\max}=\left(\dfrac{256}{1}\right)^{\frac{1}{4}}\dfrac{1}{\left(\frac{1}{2^{2s}}\right)}=4.2^{2s}$ and $u_{1,1,s}^{\min}=\left(\dfrac{1}{1}\right)^{\frac{1}{4}}\dfrac{1}{\left(\frac{1}{3^{3s}}\right)}=4.3^{3s}$, $u_{1,1,s}^{\max}=\left(\dfrac{1}{1}\right)^{\frac{1}{4}}\dfrac{1}{\left(\frac{1}{2^{3s}}\right)}=4.2^{3s}$, leading to $A^{(1)}=\emptyset$.

If $i=2$ then, $\lambda_{2,1,0}^{\min}=\dfrac{1}{3}$ and $\lambda_{2,1,0}^{\max}=\dfrac{1}{2}$, giving $u_{2,1,0}^{\min}=\left(\dfrac{1}{256}\right)^{\frac{1}{4}}\dfrac{1}{\left(\frac{1}{3}\right)}=\dfrac{3}{4}$ and $u_{2,1,0}^{\max}=\left(\dfrac{1}{256}\right)^{\frac{1}{4}}\dfrac{1}{\left(\frac{1}{2}\right)}=\dfrac{1}{2}$. Also, for all $s \in \mathbb{N}$, $\lambda_{2,2,s}^{\min}=\dfrac{1}{3^{3s}}$, $\lambda_{2,2,s}^{\max}=\dfrac{1}{2^{3s}}$ and $\lambda_{2,1,s}^{\min}=\dfrac{1}{3^{3s+1}}$, $\lambda_{2,1,s}^{\max}=\dfrac{1}{2^{3s+1}}$, giving $u_{2,2,s}^{\min}=\left(\dfrac{256}{256}\right)^{\frac{1}{4}}\dfrac{1}{\left(\frac{1}{3^{3s}}\right)}=3^{3s}$, $u_{2,2,s}^{\max}=\left(\dfrac{256}{256}\right)^{\frac{1}{4}}\dfrac{1}{\left(\frac{1}{2^{3s}}\right)}=2^{3s}$ and $u_{2,1,s}^{\min}=\left(\dfrac{1}{256}\right)^{\frac{1}{4}}\dfrac{1}{\left(\frac{1}{3^{3s+1}}\right)}=\dfrac{3^{3s+1}}{4}$, $u_{2,1,s}^{\max}=\left(\dfrac{1}{256}\right)^{\frac{1}{4}}\dfrac{1}{\left(\frac{1}{2^{3s+1}}\right)}=\dfrac{2^{3s+1}}{4}$. Thus, $A^{(2)}=\{-1\}$, and we have the increasing order $0 \leq u_{2,1,0}^{\max} \leq u_{2,1,0}^{\min} \leq 1$.

Notice that $(Df_{\zeta})^j(\theta_x,\theta_y)=(Df_{f(\zeta)})^j(\theta_x,\theta_y)=(2^j\theta_x,3^j\theta_y)$, where $j \in \mathbb{Z}$.

Applying Theorem \ref{thm:piling-corr-per}, we conclude that the piling process is
\begin{enumerate}[(0)]
\item with probability $\dfrac{1}{17}$ the bi-infinite sequence $(Z_j)_{j \in \mathbb{Z}}$ with:\begin{enumerate}[(i)]
\item entry $U.(\Theta_x,\Theta_y)$ at $j=0$
\item entries $U.(2^j\Theta_x,3^j\Theta_y)\cdot\left(\dfrac{1}{c_l}\right)^{\alpha}$ at $j=m_l+qs$ for all $l \in \mathcal{I}$ and $s \in \mathbb{N}_0$ such that $m_l+qs>0$
\item $\infty$ for all other positive indices $j$
\item $\infty$ for all negative indices $j$
\end{enumerate}
where $U$ is uniformly distributed on $[0,1]$, $\Theta$ is uniformly distributed on $\mathbb{S}^{1}$, and $U$ and $\Theta$ are independent;
\end{enumerate}
\begin{enumerate}[(I)]
\item with probability $\dfrac{16}{17}\cdot\dfrac{1}{2}=\dfrac{8}{17}$ the bi-infinite sequence $(Z_j)_{j \in \mathbb{Z}}$ with:\begin{enumerate}[(i)]
\item entry $U_0.(\Theta_x,\Theta_y)$ at $j=0$
\item entries $U_0.(2^j\Theta_x,3^j\Theta_y)\cdot\left(\dfrac{256}{c_l}\right)^{\alpha}$ at $j=m_l-m_2+qs$ for all $l \in \mathcal{I}$ and $s \in \mathbb{N}_0$ such that $m_l-m_2+qs>0$
\item $\infty$ for all other positive indices $j$
\item $\infty$ for all negative indices $j$
\end{enumerate}
where $U_0$ is uniformly distributed on $[0,1/2)$, $\Theta$ is uniformly distributed on $\mathbb{S}^{1}$, and $U_0$ and $\Theta$ are independent;
\item with probability $\dfrac{16}{17}\left(\dfrac{3}{4}-\dfrac{1}{2}\right)=\dfrac{4}{17}$ the bi-infinite sequence $(Z_j)_{j \in \mathbb{Z}}$ with:\begin{enumerate}[(i)]
\item entry $U_{1'}.(\Theta_x,\Theta_y)$ at $j=0$
\item entries $U_{1'}.(2^j\Theta_x,3^j\Theta_y)\cdot\left(\dfrac{256}{c_l}\right)^{\alpha}$ at $j=m_l-m_2+qs$ for all $l \in \mathcal{I}$ and $s \in \mathbb{N}_0$ such that $m_l-m_2+qs>0$
\item entry $U_{1'}.\left(\dfrac{1}{2}\Theta_x,\dfrac{1}{3}\Theta_y\right)\cdot 4$ at $j=-1$
\item $\infty$ for all other negative indices $j$
\end{enumerate}
where $U_{1'}$ is uniformly distributed on $[1/2,3/4)$ and $\Theta \given \{U_{1'}=z\}$ is uniformly distributed on $\left\lbrace (\theta_x,\theta_y) \in \mathbb{S}^{1}:\left\lVert\left(\dfrac{1}{2}\theta_x,\dfrac{1}{3}\theta_y\right)\right\rVert \geq \dfrac{1}{4z} \right\rbrace$;
\item with probability $\dfrac{16}{17}\left(1-\dfrac{3}{4}\right)=\dfrac{4}{17}$ the bi-infinite sequence $(Z_j)_{j \in \mathbb{Z}}$ with:\begin{enumerate}[(i)]
\item entry $U_{1}.(\Theta_x,\Theta_y)$ at $j=0$
\item entries $U_{1}.(2^j\Theta_x,3^j\Theta_y)\cdot\left(\dfrac{256}{c_l}\right)^{\alpha}$ at $j=m_l-m_2+qs$ for all $l \in \mathcal{I}$ and $s \in \mathbb{N}_0$ such that $m_l-m_2+qs>0$;
\item entry $U_{1}.\left(\dfrac{1}{2}\Theta_x,\dfrac{1}{3}\Theta_y\right)\cdot 4$ at $j=-1$
\item $\infty$ for all other negative indices $j$
\end{enumerate}
where $U_1$ is uniformly distributed on $[3/4,1]$, $\Theta$ is uniformly distributed on $\mathbb{S}^{1}$, and $U_1$ and $\Theta$ are independent.
\end{enumerate}
\end{example}
\begin{remark}\label{rmk:EI-ex-4}
For Example \ref{ex:2-exp}, we have
\begin{equation*}
\mu(U_n(\tau))=\displaystyle\sum_{i=1}^{2} \mu(B_{h_i^{-1}(u_n(\tau))}(\xi_i))=u_n(\tau)^{-\frac{1}{2}}+16u_n(\tau)^{-\frac{1}{2}}=17u_n(\tau)^{-\frac{1}{2}}.
\end{equation*}
Let $q_n=m_2+3-m_1=4$. From \cite[Corollary 5.4]{16correlated},
\begin{equation*}
\begin{split}
\mu(U_n^{(q_n)}(\tau))&=\mu(B_{h_1^{-1}(u_n(\tau))}(\xi_1))-\dfrac{1}{\det(Df_{\xi_1}^{-4})}\mu(B_{h_2^{-1}(u_n(\tau))}(\xi_2))\\
&\ +\mu(B_{h_2^{-1}(u_n(\tau))}(\xi_2)-\dfrac{1}{\det(Df_{\xi_2}^{-3})}\mu(B_{h_2^{-1}(u_n(\tau))}(\xi_2))\\
&=u_n(\tau)^{-\frac{1}{2}}-\frac{1}{16 \times 81}16u_n(\tau)^{-\frac{1}{2}}+16u_n(\tau)^{-\frac{1}{2}}-\frac{1}{8 \times 27}16u_n(\tau)^{-\frac{1}{2}}.
\end{split}
\end{equation*}
Thus, the extremal index is
\begin{equation*}
\begin{split}
\vartheta&=\lim_{n \to \infty} \dfrac{\mu(U_n^{(q_n)}(\tau))}{\mu(U_n(\tau))}\\
&=\lim_{n \to \infty} \dfrac{u_n(\tau)^{-\frac{1}{2}}-\frac{1}{16 \times 81}16u_n(\tau)^{-\frac{1}{2}}+16u_n(\tau)^{-\frac{1}{2}}-\frac{1}{8 \times 27}16u_n(\tau)^{-\frac{1}{2}}}{u_n(\tau)^{-\frac{1}{2}}+16u_n(\tau)^{-\frac{1}{2}}}=\dfrac{1370}{1377}.
\end{split}
\end{equation*}
\end{remark}

\section{A countable number of points in the same orbit}\label{sec:corr-count}
Now we consider $\mathcal{M}=\{\xi_1,\xi_2,\dots\}=\{\xi_i\}_{i \in \mathbb{N}}$ such that there exist $m_i$, $i \in \mathbb{N}$, with $\xi_i=f^{m_i}(\zeta)$, where $\zeta \in \mathcal{X}$, and $\xi_0=\lim_{i \to \infty} \xi_i$. Again, we take $m_1=0$ (\textit{i.e.} $\xi_1=\zeta$). We explore the case in which a sufficiently small neighbourhood of $\mathcal{M}$ (corresponding to the exceedance of a sufficiently high threshold $u_n(\tau)$) is a countable union of non-overlapping balls centred at each of the $\xi_i$, $i \in \mathbb{N}$. As an application of Theorem \ref{thm:piling-corr-count}, we compute the piling process for a modified version of Example 4.5 of \cite{17correlated}.
\begin{theorem}\label{thm:piling-corr-count}
Let $f$, $\mu$ and $\Psi$ be as specified in Section \ref{sec:intro} with $\mathcal{M}$ as in \ref{itm:B}. Assume that $(\mathbf{X}_n)_{n \in \mathbb{N}_0}$ has an $\alpha$-regularly varying tail, where $\alpha \in (0,1)$. Additionally, assume that $f$ is uniformly expanding along the orbit of $\zeta$. Let $U_n(\tau)=\displaystyle\bigcup_{i=1}^{\infty}B_{h_i^{-1}(u_n(\tau))}(\xi_i)$. Assume that there exists $(N(n))_{n \in \mathbb{N}}$ such that $\displaystyle\lim_{n \to \infty} N(n)=+\infty$ with $N(n)=o(n)$ and $\displaystyle\lim_{n \to \infty} \dfrac{\mu(U_n(\tau) \setminus \tilde{U}_n(\tau))}{\mu(U_n(\tau))}=0$, where $\tilde{U}_n(\tau)=\displaystyle\bigcup_{i=1}^{N(n)}B_{h_i^{-1}(u_n(\tau))}(\xi_i)$. For all $i \in \mathcal{I}$, define $p_i=\dfrac{D_ic_i^{\alpha d}}{\sum_{k=1}^{\infty}D_kc_k^{\alpha d}}$. Let $A^{(i)}$ be as defined in \ref{itm:A1}. If $A^{(i)}=\emptyset$, the piling process is
\begin{enumerate}[(0)]
\item\label{itm:(0)-piling-fin-nper} with probability $p_i$ the bi-infinite sequence $(Z_j)_{j \in \mathbb{Z}}$ with:
\begin{enumerate}[(i)]
\item entry $U.\Theta$ at $j=0$;
\item entries $U.Df_{\xi_i}^{m_l-m_i}(\Theta)\left(\dfrac{c_i}{c_l}\right)^{\alpha}$ at $j=m_l-m_i$ for all $l \geq i+1$;
\item $\infty$ for all other positive indices $j$;
\item $\infty$ for all negative indices $j$;
\end{enumerate}
where $U$ is uniformly distributed on $[0,1]$, $\Theta$ is uniformly distributed on $\mathbb{S}^{d-1}$, and $U$ and $\Theta$ are independent.
\end{enumerate}
If $A^{(i)} \neq \emptyset$, assume there exists an increasing ordering of the $u_{i,l}^{min,max}$ such that $u_{i,l_p}^{min} \leq u_{i,l_{p+1}}^{max}$ for all $p \in \{1,\dots,\#A^{(i)}-1\}$. Then, the piling process is
\begin{enumerate}[(I)]
\item with probability $p_iu_{i,l_1}^{max}$ the bi-infinite sequence $(Z_j)_{j \in \mathbb{Z}}$ with:\begin{enumerate}[(i)]
\item entry $U_0.\Theta$ at $j=0$;
\item entries $U_0.Df_{\xi_i}^{m_l-m_i}(\Theta)\left(\dfrac{c_i}{c_l}\right)^{\alpha}$ at $j=m_l-m_i$ for all $l \geq i+1$;
\item $\infty$ for all other positive indices $j$;
\item $\infty$ for all negative indices $j$;
\end{enumerate}
where $U_0$ is uniformly distributed on $[0,u_{i,l_1}^{max})$, $\Theta$ is uniformly distributed on $\mathbb{S}^{d-1}$, and $U_0$ and $\Theta$ are independent;
\item with probability $p_i(u_{i,l_p}^{min}-u_{i,l_p}^{max})$, where $p \in \{1,\dots,\#A^{(i)}\}$, the bi-infinite sequence $(Z_j)_{j \in \mathbb{Z}}$ with:\begin{enumerate}[(i)]
\item entry $U_{p'}.\Theta$ at $j=0$;
\item entries $U_{p'}.Df_{\xi_i}^{m_l-m_i}(\Theta)\left(\dfrac{c_i}{c_l}\right)^{\alpha}$ at $j=m_l-m_i$ for all $l \geq i+1$;
\item $\infty$ for all other positive indices $j$;
\item entries $U_{p'}.Df_{\xi_i}^{m_l-m_i}(\Theta)\left(\dfrac{c_i}{c_l}\right)^{\alpha}$ at $j=m_l-m_i$ for all $l \in \{l_1,\dots,l_p\}$;
\item $\infty$ for all other negative indices $j$;
\end{enumerate}
where $U_{p'}$ is uniformly distributed on $[u_{i,l_p}^{max},u_{i,l_p}^{min})$ and $\Theta \given \{U_{p'}=u\}$ is uniformly distributed on $\left\lbrace w \in \mathbb{S}^{d-1}:\norm{Df_{\xi_i}^{m_{l_p}-m_i}(w)} \geq \dfrac{1}{u}\left(\dfrac{c_{l_p}}{c_i}\right)^{\alpha} \right\rbrace$;
\item with probability $p_i(u_{i,l_{p+1}}^{max}-u_{i,l_p}^{min})$, where $p \in \{1,\dots,\#A^{(i)}-1\}$, the bi-infinite sequence $(Z_j)_{j \in \mathbb{Z}}$ with:\begin{enumerate}[(i)]
\item entry $U_p.\Theta$ at $j=0$;
\item entries $U_p.Df_{\xi_i}^{m_l-m_i}(\Theta)\left(\dfrac{c_i}{c_l}\right)^{\alpha}$ at $j=m_l-m_i$ for all $l \geq i+1$;
\item $\infty$ for all other positive indices $j$;
\item entries $U_p.Df_{\xi_i}^{m_l-m_i}(\Theta)\left(\dfrac{c_i}{c_l}\right)^{\alpha}$ at $j=m_l-m_i$ for all $l \in \{l_1,\dots,l_p\}$;
\item $\infty$ for all other negative indices $j$;
\end{enumerate}
where $U_p$ is uniformly distributed on $[u_{i,l_p}^{min},u_{i,l_{p+1}}^{max})$, $\Theta$ is uniformly distributed on $\mathbb{S}^{d-1}$, and $U_p$ and $\Theta$ are independent;
\item with probability $p_i.(1-u_{i,l_{\#A^{(i)}}}^{min})$ the bi-infinite sequence $(Z_j)_{j \in \mathbb{Z}}$ with:\begin{enumerate}[(i)]
\item entry $U_{\#A^{(i)}}.\Theta$ at $j=0$;
\item entries $U_{\#A^{(i)}}.Df_{\xi_i}^{m_l-m_i}(\Theta)\left(\dfrac{c_i}{c_l}\right)^{\alpha}$ at $j=m_l-m_i$ for all $l \geq i+1$;
\item $\infty$ for all other positive indices $j$;
\item entries $U_{\#A^{(i)}}.Df_{\xi_i}^{m_l-m_i}(\Theta)\left(\dfrac{c_i}{c_l}\right)^{\alpha}$ at $j=m_l-m_i$ for all $l \in A^{(i)}$;
\item $\infty$ for all other negative indices $j$;
\end{enumerate}
where $U_{\#A^{(i)}}$ is uniformly distributed on $[u_{i,l_{\#A^{(i)}}}^{min},1]$, $\Theta$ is uniformly distributed on $\mathbb{S}^{d-1}$, and $U_{\#A^{(i)}}$ and $\Theta$ are independent.
\end{enumerate}
\end{theorem}
\begin{proof}
Recall that $(Z_j)_{j \in \mathbb{Z}}$ has the distribution of $(Y_j)_{j \in \mathbb{Z}}$ conditional on $\displaystyle\inf_{j \leq -1} \norm{Y_j} \geq 1$ (see Definition \ref{def:piling}). The result is a consequence of Theorem \ref{thm:piling-corr-nper} in the sense that the use of the sequence $(N(n))_{n \in \mathbb{N}}$ implies that, at each $n$, $\mathcal{M}$ is as in \ref{itm:A} (for $\mathcal{I}=\{1,\dots,N(n)\}$).

\underline{\textbf{Step 1}} We check that the process $(Y_j)_{j \in \mathbb{Z}}$ is, with probability $p_i$, the bi-infinite sequence with:
\begin{enumerate}[(i)]
\item entry $U.\Theta$ at $j=0$;
\item entries $U.Df_{\xi_i}^{m_l-m_i}(\Theta)\left(\dfrac{c_i}{c_l}\right)^{\alpha}$ at $j=m_l-m_i$, for all $l \geq i+1$;
\item $\infty$ for all other positive indices $j$;
\item $\infty$ for all negative indices $j$ except, possibly, $U.Df_{\xi_i}^{m_l-m_i}(\Theta)\left(\dfrac{c_i}{c_l}\right)^{\alpha}$ at $j=m_l-m_i$ for $l=1,\dots,i-1$;
\end{enumerate}
where $U$ is uniformly distributed on $[0,1]$, $\Theta$ is uniformly distributed on $\mathbb{S}^{d-1}$, and $U$ and $\Theta$ are independent.

Verifying that condition (2) in Definition \ref{def:piling} is satisfied for $(Y_j)_{j \in \mathbb{Z}}$ as just described is straightforward. The (a.s.) positive exponential growth of $\norm{Df_{\xi_i}^j(\Theta)}$ for positive $j$ leads to (3) being satisfied. Since $p_1>0$, with positive probability all negatively indexed entries are equal to $\infty$ so (4) is satisfied.

We are left to check that condition (1) in Definition \ref{def:piling} holds.

\underline{\textbf{Sub-step 1.1}} $p_i$ is the probability that an exceedance of the threshold $u_n(\tau)$ by $\norm{\mathbf{X}_{r_n}}$ is due to a hit (at time $r_n$) to the ball around $\xi_i$ of radius $h_i^{-1}(u_n(\tau))$.

Analogous argument to that in Sub-step 1.1 in the proof of Theorem \ref{thm:piling-corr-nper}: how much does a neighbourhood around $\xi_i$ (corresponding to the exceedance of a threshold $u_n(\tau)$) weights in a neighbourhood of $\mathcal{M}$ (corresponding to the exceedance of the same $u_n(\tau)$). Observe that
\begin{equation*}
\displaystyle\lim_{n \to \infty}\sum_{i=N(n)+1}^{\infty}p_i=\displaystyle\lim_{n \to \infty} \dfrac{\mu(U_n(\tau) \setminus \tilde{U}_n(\tau))}{\mu(U_n(\tau))}=0
\end{equation*}
so that $\displaystyle\sum_{i=1}^{\infty}p_i<\infty$ (and, in particular, equal to $1$).

\underline{\textbf{Sub-step 1.2}} Assume a hit at time $r_n$ to the ball around $\xi_i$ of radius $h_i^{-1}(u_n(\tau))$. Then, $(Y_j)_{j \in \mathbb{Z}}$ is as described by (i)-(iv) in Step 1.

Let $\{\norm{\mathbf{X}_{r_n}}>u_n(\tau)\}=f^{-r_n}(\tilde{U}_n(\tau))$. Then, by Theorem \ref{thm:piling-corr-nper}, we have that $(Y_j)_{j \in \mathbb{Z}}$ is, with probability $p_i$, the bi-infinite sequence with:
\begin{enumerate}[(i)]
\item entry $U.\Theta$ at $j=0$;
\item entries $U.Df_{\xi_i}^{m_l-m_i}(\Theta)\left(\dfrac{c_i}{c_l}\right)^{\alpha}$ at $j=m_l-m_i$, for all $l=i+1,\dots,N(n)$;
\item $\infty$ for all other positive indices $j$;
\item $\infty$ for all negative indices $j$ except, possibly, $U.Df_{\xi_i}^{m_l-m_i}(\Theta)\left(\dfrac{c_i}{c_l}\right)^{\alpha}$ at $j=m_l-m_i$ for $l=1,\dots,i-1$;
\end{enumerate}
where $U$ is uniformly distributed on $[0,1]$, $\Theta$ is uniformly distributed on $\mathbb{S}^{d-1}$, and $U$ and $\Theta$ are independent.

Since $\tilde{U}_n(\tau) \sim U_n(\tau)$ (a.s.), by letting $N(n) \to \infty$, we conclude that $(Y_j)_{j \in \mathbb{Z}}$ is as described by (i)-(iv) in Step 1.

\underline{\textbf{Step 2}} The distribution of $(Z_j)_{j \in \mathbb{Z}}$ is given by (0)-(IV).

Analogous to Theorem \ref{thm:piling-corr-nper}, accounting for the changes in $(Y_j)_{j \in \mathbb{Z}}$ that have already been discussed.
\end{proof}
\begin{example}\label{ex:psi-ncont}
Let $f(x)=3x \mod 1$, $x \in [0,1]$, and $\mu= $ Lebesgue measure on $[0,1]$ (invariant for $f$). Take $\zeta=\displaystyle\sum_{j=1}^{+\infty}\left(\dfrac{1}{3}\right)^{3^j}$, and define the observable $\psi$ as
\begin{equation*}
\psi(x):=\begin{cases}
\modl{x-\zeta}^{-2}, x \in B_{\varepsilon_1}(\zeta)\\
\dfrac{1}{2^{i-1}}\modl{x-f^{3^{i-1}}(\zeta)}^{-2}, \ x \in B_{\varepsilon_i}(f^{3^{i-1}}), \ i=2,3,\dots\\
0, \ \text{otherwise}
\end{cases}
\end{equation*}
where $\varepsilon_i>0$ for all $i \in \mathbb{N}$. Observe that, presented as in (\ref{eq:psi}),
\begin{equation*}
\psi(x)=\displaystyle\sum_{i=1}^{\infty}h_i(\modl{x-\xi_i})\mathbf{1}_{B_{\varepsilon_i}(\xi_i)}(x)
\end{equation*}
for $\xi_1=\zeta$, $\xi_i=f^{3^{i-1}}(\zeta)$ for $i=2,3,\dots$, $\xi_0=\displaystyle\lim_{i \to \infty} \xi_i=0$, and $h_i(t)=\dfrac{1}{2^{i-1}}t^{-2}$ for all $i \in \mathbb{N}$ (so that $\alpha=1/2$). In particular, $\mathcal{M}=\{\xi_i\}_{i \in \mathbb{N}}$ and equation (\ref{eq:alpha-hea}) holds with $a_n=(24+16\sqrt{2})n^2$.

Since $\mu=\text{Lebesgue}$, we have that $D_i=1$ for all $i \in \mathbb{N}$. Also, $d=1$ and $c_i=\dfrac{1}{2^{i-1}}$ for all $i \in \mathbb{N}$. Thus, $p_i=\dfrac{\left(\dfrac{1}{2^{i-1}}\right)^{\frac{1}{2}}}{\displaystyle\sum_{k=1}^{\infty}\left(\dfrac{1}{2^{k-1}}\right)^{\frac{1}{2}}}=\dfrac{\left(\dfrac{1}{\sqrt{2}}\right)^{i-1}}{\displaystyle\sum_{k=1}^{\infty}\left(\dfrac{1}{\sqrt{2}}\right)^{k-1}}=\left(1-\dfrac{1}{\sqrt{2}}\right)\left(\dfrac{1}{\sqrt{2}}\right)^{i-1}$ for all $i \in \mathbb{N}$.

Because
\begin{equation*}
\displaystyle\lim_{n \to \infty} \dfrac{\mu(U_n(\tau) \setminus \tilde{U}_n(\tau))}{\mu(U_n(\tau))}=\displaystyle\lim_{n \to \infty}\sum_{i=N(n)+1}^{\infty}p_i=\displaystyle\lim_{n \to \infty}\left(1-\dfrac{1}{\sqrt{2}}\right)\left(\dfrac{1}{\sqrt{2}}\right)^{N(n)}=0
\end{equation*}
for any $N(n)$ such that $\lim_{n \to \infty} N(n)=+\infty$ with $N(n)=o(n)$, then $(N(n))_{n \in \mathbb{N}}$ as in the statement of Theorem \ref{thm:piling-corr-count} exists. For example, let $N(n)=\log(n)$.

We have $f'(x)=3$ for all $x \in [0,1]$, so that $(f^{-j})'(\xi_i)=\dfrac{1}{3^j}$ when $j<0$ and, if $i>1$, $\left(\dfrac{c_l}{c_i}\right)^{\alpha} \geq \sqrt{2}$ for all $l=1,\dots,i-1$. Therefore, $A^{(i)}=\emptyset$ for all $i \in \mathbb{N}$. Applying Theorem \ref{thm:piling-corr-count}, we conclude that the piling process is with probability $p_i=\left(1-\dfrac{1}{\sqrt{2}}\right)\left(\dfrac{1}{\sqrt{2}}\right)^{i-1}$ the bi-infinite sequence $(Z_j)_{j \in \mathbb{Z}}$ with:
\begin{enumerate}[(i)]
\item entry $U$ at $j=0$;
\item entries $U.3^{m_l-m_i}.\left(\dfrac{c_i}{c_l}\right)^{\frac{1}{2}}$ at $j=m_l-m_i$ for all $l \geq i+1$;
\item $\infty$ for all other positive indices $j$;
\item $\infty$ for all negative indices $j$;
\end{enumerate}
where $U$ is uniformly distributed on $[0,1]$, $i \in \mathbb{N}$.
\end{example}
\begin{remark}\label{rmk:EI-ex-5}
For Example \ref{ex:psi-ncont}, we have
\begin{equation*}
\mu(\tilde{U}_n(\tau))=\displaystyle\sum_{i=1}^{N(n)} \mu(B_{h_i^{-1}(u_n(\tau))}(\xi_i))=\displaystyle\sum_{i=1}^{N(n)} 2.\left(\dfrac{1}{\sqrt{2}}\right)^{i-1}u_n(\tau)^{-\frac{1}{2}}=2u_n(\tau)^{-\frac{1}{2}}\dfrac{1-\left(\dfrac{1}{\sqrt{2}}\right)^{N(n)}}{1-\dfrac{1}{\sqrt{2}}}.
\end{equation*}
Let $q_n=N(n)$. From \cite[Corollary 4.5]{16correlated},
\begin{equation*}
\begin{split}
\mu(\tilde{U}_n^{(q_n)}(\tau))&=\displaystyle\sum_{i=1}^{N(n)-1} \left(\mu(B_{h_i^{-1}(u_n(\tau))}(\xi_i))-\dfrac{1}{3^3.3^{3^{i-1}}}\mu(B_{h_{i+1}^{-1}(u_n(\tau))}(\xi_{i+1}))\right)\\
&\quad +\mu(B_{h_{N(n)}^{-1}(u_n(\tau))}(\xi_{N(n)}))\\
&=\displaystyle\sum_{i=1}^{N(n)-1} \left(2.\left(\dfrac{1}{\sqrt{2}}\right)^{i-1}u_n(\tau)^{-\frac{1}{2}}-\dfrac{1}{3^3.3^{3^{i-1}}}2.\left(\dfrac{1}{\sqrt{2}}\right)^{i}u_n(\tau)^{-\frac{1}{2}}\right)\\
&\quad +2.\left(\dfrac{1}{\sqrt{2}}\right)^{N(n)-1}u_n(\tau)^{-\frac{1}{2}}\\
&=2u_n(\tau)^{-\frac{1}{2}}\left[\displaystyle\sum_{i=1}^{N(n)-1} \left(\left(\dfrac{1}{\sqrt{2}}\right)^{i-1}-\dfrac{1}{3^33^{3^{i-1}}}\left(\dfrac{1}{\sqrt{2}}\right)^{i}\right)+\left(\dfrac{1}{\sqrt{2}}\right)^{N(n)-1}\right].
\end{split}
\end{equation*}
Thus, the extremal index is
\begin{equation*}
\begin{split}
\vartheta&=\lim_{n \to \infty} \dfrac{\mu(\tilde{U}_n^{(q_n)}(\tau))}{\mu(\tilde{U}_n(\tau))} \approx 0.997242
\end{split}
\end{equation*}
round off after stabilisation of decimal places (using numerical computation).
\end{remark}

\appendix
\section{Dependence requirements}\label{appx:dependance}
In order to include in this text the conditions $\D_{q_n}$ and $\D'_{q_n}$ from \cite{20enriched}, we need some formalism from sections 3, 3.1 and 3.2 of the same paper that we now summarise.

We require sequences $(k_n)_{n \in \mathbb{N}}, (r_n)_{n \in \mathbb{N}}, (t_n)_{n \in \mathbb{N}}$ such that $$k_n,r_n,t_n \underset{\scriptsize{n \to \infty}}{\longrightarrow}\infty \text{ and } k_nt_n=o(n),$$
where $r_n:=\floor{n/k_n}$, and a sequence $(q_n)_{n \in \mathbb{N}}$ satisfying
$$q_n=o(r_n).$$
These sequences play a role in identifying clusters of exceedances by two different approaches (though made equivalent by condition $\D'_{q_n}$): the first consists on dividing the observations into $k_n$ blocks, each of size $r_n$, and with time gaps of size $t_n$ between subsequent blocks, and determining that exceedances in different blocks belong to different clusters; the second determines that a cluster ends when no exceedances are registered in $q_n$ successive time instants. In particular, notice that $q_n=q$ for all $n \in \mathbb{N}$, where $q \in \mathbb{N}$, is a possibility, which relates to the case of prime period $q$ periodicity.

Indeed, the conditions $\D_{q_n}$ and $\D'_{q_n}$ can be seen, respectively, as long range and short range independence requirements on the cluster point process: whereas condition $\D_{q_n}$ expresses that, has time run for long enough, blocks that are sufficiently far apart are basically independent, we have that $\D'_{q_n}$ essentially imposes that not more than one cluster of exceedances is expected within the same block.

We resume with more formalism in the following paragraphs.

Let $\mathcal{V}=\mathbb{R}^d$ (with the Euclidean norm), and $\mathcal{V}^{\mathbb{N}_0}$/$\mathcal{V}^{\mathbb{Z}}$ be the spaces of one-sided/two-sided $\mathcal{V}$-valued sequences. Take the (one-sided/two-sided) shift map $\sigma:\mathcal{V}^{\mathbb{N}_0,\mathbb{Z}} \to \mathcal{V}^{\mathbb{N}_0,\mathbb{Z}}$. We identify our stationary sequence $(\mathbf{X}_n)_{n \in \mathbb{N}_0}$, which takes values in $\mathbb{R}^d$, with the coordinate variable process in $(\mathcal{V}^{\mathbb{N}_0},\mathcal{B}^{\mathbb{N}_0},\mathbb{P})$ given by Kolmogorov's extension theorem where $\mathcal{B}^{\mathbb{N}_0}$ denotes the product $\sigma$-algebra, in other words, $\mathcal{B}^{\mathbb{N}_0}$ is the $\sigma$-algebra generated by the coordinate functions $Z_n:\mathcal{V}^{\mathbb{N}_0} \to \mathcal{V}$ such that $Z_n(x_0,x_1,\dots)=x_n$, for all $n \in \mathbb{N}_0$. Observe that $Z_i=\sigma \circ Z_{i-1}$, for all $i \in \mathbb{N}$, and the stationarity of $(\mathbf{X}_n)_{n \in \mathbb{N}_0}$ results in $\sigma$-invariance of $\mathbb{P}$.

Let
\begin{equation*}
\mathscr{F}:=\left\lbrace\{(x_j)_j \in \mathcal{V}^{\mathbb{N}_0,\mathbb{Z}}: x_j \in H_j, j=0,\dots,m\}: H_j \in \mathcal{F}_{\mathcal{V}}, j=0,\dots,m, \ m \in \mathbb{N}\right\rbrace
\end{equation*}
where $\mathcal{F}_{\mathcal{V}}$ denotes the field generated by the rectangles of $\mathcal{V}$ of the form $[e_1,f_1) \times \dots \times [e_d,f_d)$ (in particular, $\mathscr{F}$ is a field).

For each $l=1,\dots,m$, $m \in \mathbb{N}$, assume that $A_l \in \mathscr{F}$ and define
\begin{equation*}
A_{n,l}=\left\lbrace\left(u_n^{-1}(\norm{\mathbf{X}_j})\dfrac{\mathbf{X}_j}{\norm{\mathbf{X}_j}}\right)_j \in \mathcal{V}^{\mathbb{N}_0}: \left(u_n^{-1}(\norm{\mathbf{X}_j})\dfrac{\mathbf{X}_j}{\norm{\mathbf{X}_j}}\right)_j \in A_l \right\rbrace
\end{equation*}
where $u_n^{-1}$ is as in (\ref{eq:gen-u_n^{-1}}). We observe that
\begin{equation}
\left\lVert u_n^{-1}(\norm{\mathbf{X}_{j}})\dfrac{\mathbf{X}_{j}}{\norm{\mathbf{X}_{j}}} \right\rVert < \tau \iff \norm{\mathbf{X}_{j}} > u_n(\tau).
\end{equation}

We will be particularly interested in the events
\begin{equation*}
A_{n,l}^{(q_n)}:=A_{n,l} \cap \bigcap_{j=1}^{q_n} \sigma^{-j}(A_{n,l})^{c}, \quad l=1,\dots,m,
\end{equation*}
and
\begin{equation*}
U_n(\tau):=\{\norm{X_0}>u_n(\tau)\}
\end{equation*}
for $q_n$ and $u_n$ as defined before.

Just as $U_n(\tau)$, $A_{n,l}^{(q_n)}$ represents a set of exceedances but with the requirement that a pattern with length $l$ that has just been observed is not to be repeated for the next $q_n$ time instants.

Finally, let $J_l=[a_l,b_l)$ where $0 \leq a_1<b_1 \leq a_2<b_2 \leq \dots \leq a_m<b_m \leq 1$ and, for each $n \in \mathbb{N}$, define $J_{n,l}:=[(\ceil{k_na_l-1})r_n,(\floor{k_nb_l+1})r_n)$ where $k_n$ and $r_n$ are as above.

For an interval $I$ contained in $[0,+\infty)$, we use the notation $\mathscr{W}_{I}(A) \equiv \displaystyle\bigcap_{j \in I \cap \mathbb{N}_0} \sigma^{-j}(A^c)$ and $\mathscr{W}_{I}^c(A) \equiv (\mathscr{W}_{I}(A))^c$.
 
\textbf{Condition $\D_{q_n}$}. We say that $\D_{q_n}$ holds for the sequence $\mathbf{X}_0,\mathbf{X}_1,\dots$ if there exist sequences $(k_n)_{n \in \mathbb{N}}, (r_n)_{n \in \mathbb{N}}, (t_n)_{n \in \mathbb{N}}$ and $(q_n)_{n \in \mathbb{N}}$ as above, such that for every $m,t,n \in \mathbb{N}$ and every $J_l$ and $A_l$, with $l=1,\dots,m$, we have
\begin{equation*}
\left\lvert \mathbb{P}\left(A_{n,l}^{(q_n)} \cap \bigcap_{i=l}^{m} \mathscr{W}_{J_{n,i}}(A_{n,i}^{(q_n)})\right) - \mathbb{P}(A_{n,l}^{(q_n)}) \mathbb{P}\left(\bigcap_{i=l}^{m} \mathscr{W}_{J_{n,i}}(A_{n,i}^{(q_n)})\right) \right\rvert \leq \gamma(n,t)
\end{equation*}
where $\min\{J_{n,l} \cap \mathbb{N}_0\} \geq t$ and $\gamma(n,t)$ is decreasing in $t$ for each $n$ and $\lim_{n \to \infty} n\gamma(n,t_n)=0$.

\textbf{Condition $\D'_{q_n}$}. We say that $\D'_{q_n}$ holds for the sequence $\mathbf{X}_0,\mathbf{X}_1,\dots$ if there exist sequences $(k_n)_{n \in \mathbb{N}}, (r_n)_{n \in \mathbb{N}}, (t_n)_{n \in \mathbb{N}}$ and $(q_n)_{n \in \mathbb{N}}$ as above, such that for every $A_1 \in \mathscr{F}$, we have
\begin{equation*}
\lim_{n \to \infty} n\mathbb{P}\left(A_{n,1}^{(q_n)} \cap \mathscr{W}^c_{[q_n+1,r_n)}(A_{n,1})\right)=0.
\end{equation*}

\textbf{Condition $\tilde{\D'}_{q_n}$}. We say that $\tilde{\D'}_{q_n}$ holds for the sequence $\mathbf{X}_0,\mathbf{X}_1,\dots$ if there exist sequences $(k_n)_{n \in \mathbb{N}}, (r_n)_{n \in \mathbb{N}}, (t_n)_{n \in \mathbb{N}}$ and $(q_n)_{n \in \mathbb{N}}$ as above, such that for every $\tau>0$, we have
\begin{equation*}
\lim_{n \to \infty} n\mathbb{P}\left(U_n(\tau) \cap \mathscr{W}^c_{[q_n+1,r_n)}(U_n(\tau))\right)=0.
\end{equation*}
\begin{remark}
Condition $\tilde{\D'}_{q_n}$ implies Condition $\D'_{q_n}$ and should be easier to check.
\end{remark}

As pointed out in section 4 of \cite{20enriched}, $\D_{q_n}$ and $\D'_{q_n}$ follow from decay of correlations, for processes arising from dynamical systems.
\begin{definition}
Let $\mathcal{C}_1$ and $\mathcal{C}_2$ be Banach spaces of $\mathbb{R}$-valued measurable functions defined on $\mathcal{X}$. Let the \textit{correlation} between non-zero functions $\phi \in \mathcal{C}_1$ and $\psi \in \mathcal{C}_2$ with respect to $\mu$ at time $n \in \mathbb{N}$ be defined as
\begin{equation*}
\text{Cor}_{\mu}(\phi,\psi,n):=\dfrac{1}{\norm{\phi}_{\mathcal{C}_1}\norm{\psi}_{\mathcal{C}_2}} \left\lvert \int \phi(\psi \circ T^n) d\mu - \int \phi d\mu \int \psi d\mu \right\rvert.
\end{equation*}
The system is said to have \textit{decay of correlations}, with respect to $\mu$, for observables in $\mathcal{C}_1$ \textit{against} observables in $\mathcal{C}_2$ if there exists a rate function $\rho:\mathbb{N} \to [0,+\infty)$ with $\displaystyle\lim_{n \to \infty} \rho(n)=0$ and such that for all $\phi \in \mathcal{C}_1$ and for all $\psi \in \mathcal{C}_2$ it holds $\text{Cor}_{\mu}(\phi,\psi,n) \leq \rho(n)$.
\end{definition}
The following lemma gives the common shortcut to check that $\D_{q_n}$ and $\D'_{q_n}$ hold. It has analogous counterparts for weaker versions of the conditions $\D_{q_n}$ and $\D'_{q_n}$ that have already been established (for example, the versions compatible with extreme value laws or rare events point processes).
\begin{lemma}\label{lem:conditions}
If the system has decay of correlations against $L^1$ and
\begin{enumerate}[(1)]
\item $\displaystyle\lim_{n \to \infty} \norm{\mathbf{1}_{A_{n,l}^{(q_n)}}}_{\mathcal{C}_1}n\rho(t_n)=0$ for some sequence $(t_n)_n$ with $t_n=o(n)$;
\item $\displaystyle\lim_{n \to \infty} \norm{\mathbf{1}_{U_n(\tau)}}_{\mathcal{C}_1}\sum_{j=q_n}^{n}\rho(j)=0$;
\end{enumerate}
are satisfied, then $\D_{q_n}$ and $\D'_{q_n}$ hold.
\end{lemma}
\begin{proof}
We first check that (1) implies $\D_{q_n}$. Taking $\phi=\mathbf{1}_{A_{n,l}^{(q_n)}}$ and $\psi=\mathbf{1}_{T^{-t}\left(\bigcap_{i=l}^m W_{J_{n,i}}(A_{n,i}^{(q_n)})\right)}$, we have
\begin{equation*}
\begin{split}
&\left\lvert \mathbb{P}\left(A_{n,l}^{(q_n)} \cap \bigcap_{i=l}^{m} W_{J_{n,i}}(A_{n,i}^{(q_n)})\right) - \mathbb{P}(A_{n,l}^{(q_n)}) \mathbb{P}\left(\bigcap_{i=l}^{m} W_{J_{n,i}}(A_{n,i}^{(q_n)})\right) \right\rvert\\
&=\left\lvert \int \phi(\psi \circ T^t) d\mu - \int \phi d\mu \int \psi d\mu\right\rvert \leq \norm{\mathbf{1}_{A_{n,l}^{(q_n)}}}_{\mathcal{C}_1}\rho(t)
\end{split}
\end{equation*}
so that, by (1), $\D_{q_n}$ holds with $\gamma(n,t)=\norm{\mathbf{1}_{A_{n,l}^{(q_n)}}}_{\mathcal{C}_1}\rho(t)$. In turn, taking $\phi=\psi=\mathbf{1}_{U_n(\tau)}$, and since
\begin{equation*}
\mathbb{P}\left(U_n(\tau) \cap T^{-j}(U_n(\tau))\right)=\int \phi(\phi \circ T^j) d\mu \leq (\mu(U_n(\tau)))^2+\norm{\mathbf{1}_{U_n(\tau)}}_{\mathcal{C}_1}\mu(U_n(\tau))\rho(j)
\end{equation*}
which implies that
\begin{equation*}
\begin{split}
n\sum_{j=q_n+1}^{r_n-1}\mathbb{P}\left(U_n(\tau) \cap T^{-j}(U_n(\tau))\right) &\leq n(r_n-1)(\mu(U_n(\tau)))^2 + n.\norm{\mathbf{1}_{U_n(\tau)}}_{\mathcal{C}_1}\mu(U_n(\tau))\sum_{j=q_n+1}^{r_n-1}\rho(j)\\
&\leq \dfrac{n^2(\mu(U_n(\tau)))^2}{k_n}+n\norm{\mathbf{1}_{U_n(\tau)}}_{\mathcal{C}_1}\mu(U_n(\tau))\sum_{j=q_n}^{n}\rho(j)\\
&\leq \dfrac{\tau^2}{k_n}+\tau\norm{\mathbf{1}_{U_n(\tau)}}_{\mathcal{C}_1}\sum_{j=q_n}^{n}\rho(j) \underset{\scriptsize{n \to \infty}}{\longrightarrow}0
\end{split}
\end{equation*}
using (2). Therefore, we have that condition $\tilde{\D'}_{q_n}$ is satisfied, which implies $\D'_{q_n}$.
\end{proof}

\subsection{Application to the Examples in Sections \ref{sec:corr-finite} and \ref{sec:corr-count}}\label{appx:dependance-ex}
We check that the Examples presented in Sections \ref{sec:corr-finite} and \ref{sec:corr-count} meet the dependence requirements $\D_{q_n}$ and $\D'_{q_n}$. For that matter, we clearly are interested in proving that (1) and (2) of Lemma \ref{lem:conditions} hold for the same examples. In fact, it is enough that the system has summable decay of correlations against $L^1(\mu)$ and that there exists $C>0$ such that, for all $n \in \mathbb{N}$, $\mathbf{1}_{A_{n,l}^{(q_n)}},\mathbf{1}_{U_n(\tau)} \in \mathcal{C}_1$ and $\norm{\mathbf{1}_{A_{n,l}^{(q_n)}}}_{\mathcal{C}_1},\norm{\mathbf{1}_{U_n(\tau)}}_{\mathcal{C}_1} \leq C$ for (1) and (2) of Lemma \ref{lem:conditions} to follow. For all the one dimensional examples illustrating both the finite and countable maximal set scenarios, we indeed have exponential decay of correlations for $BV$ against $L^1$ observables. For completeness of the exposition, we recall the definition of the Banach space $BV$.
\begin{definition}
Let $\phi:I \to \mathbb{R}$ be a measurable function, where $I \subset \mathbb{R}$ is an interval. Let the \textit{variation} of $\phi$ be defined as
\begin{equation*}
Var(\phi):=\sup\left\lbrace \sum_{i=0}^{n-1}\modl{\phi(x_{i+1})-\phi(x_i)} \right\rbrace
\end{equation*}
where the supremum is taken over all ordered sequences $(x_i)_{i=0}^{n}$ in $I$. The \textit{$BV$-norm} is defined as $\norm{\phi}_{BV}:=\sup\modl{\phi}+Var(\phi)$ and the space $BV:=\{\phi:I \to \mathbb{R}:\norm{\phi}_{BV}<\infty\}$ (equipped with the $BV$-norm) is a Banach space, the space of functions with \textit{bounded variation}.
\end{definition}
It is immediate to notice that the $BV$-norm of $\mathbf{1}_A$ is bounded above by
$$1+2\#\{\text{connected components of A}\}$$
for any measurable $A \subset I$. In particular, when the maximal set is finite, since the sets $A_{n,l}^{(q_n)}$ and $U_n(\tau)$ correspond, respectively, to a finite number of annuli or balls around the maximal points then $\norm{\mathbf{1}_{A_{n,l}^{(q_n)}}}_{BV},\norm{\mathbf{1}_{U_n(\tau)}}_{BV} \leq C$ for some $C>0$ that doesn't depend on $n \in \mathbb{N}$. Therefore, in case $\mathcal{M}$ is finite we are done with proving that (1) and (2) of Lemma \ref{lem:conditions} hold.

In the countable setting, however, we do not expect a uniform bound on the $BV$-norms of the relevant indicator functions. Still, we can prove what we had wished for.
\begin{lemma}\label{lem:shortcut-corr-ncont}
Let $f$, $\mu$, $\mathcal{M}$ and $\psi$ be as in Example \ref{ex:psi-ncont}. Then
\begin{enumerate}[(1)]
\item $\displaystyle\lim_{n \to \infty} \norm{\mathbf{1}_{\tilde{A}_{n,l}^{(q_n)}}}_{\mathcal{C}_1}n\rho(t_n)=0$ for some sequence $(t_n)_n$ with $t_n=o(n)$;
\item $\displaystyle\lim_{n \to \infty} \norm{\mathbf{1}_{\tilde{U}_n(\tau)}}_{\mathcal{C}_1}\sum_{j=q_n}^{n}\rho(j)=0$.
\end{enumerate}
are satisfied, where $\tilde{U}_n(\tau)=\tilde{U}_n$ as in Example \ref{ex:psi-ncont} and the sets $\tilde{A}_{n,l}^{(q_n)}$ are to be characterised in the proof.
\end{lemma}
\begin{proof}
Here $\mathcal{C}_1=BV$ and $\rho(t)=\left(\dfrac{1}{3}\right)^{t}$. Setting $q_n=N(n)$ then (2) rewrites
\begin{equation*}
\displaystyle\lim_{n \to \infty} \norm{\mathbf{1}_{\tilde{U}_n(\tau)}}_{BV}\left(\dfrac{1}{3}\right)^{N(n)} \leq \displaystyle\lim_{n \to \infty} (1+2N(n))\left(\dfrac{1}{3}\right)^{N(n)}=0.
\end{equation*}
As for (1), we take $A_l=[0,\tau_0) \times \dots \times [0,\tau_m)$ where $\tau_0 \leq \tau_1 \leq \dots \leq \tau_m$ (\textit{i.e.} $H_j=[0,\tau_j)$ for all $j=0,\dots,m$). Then, $\tilde{A}_{n,l}^{(q_n)}:=\displaystyle\bigcap_{j=0}^{l} f^{-j}(\tilde{U}_n(\tau_j)) \cap \displaystyle\bigcap_{j=1}^{q_n}\bigcup_{k=j}^{l} f^{-k}(\tilde{U}_n^c(\tau_{k-j}))$. We claim that our choice of $A_l$ does provide us with the biggest possible number of connected components for $\tilde{A}_{n,l}^{(q_n)}$. Now, $\displaystyle\bigcap_{j=0}^{l} f^{-j}(\tilde{U}_n(\tau_j))$ consists of $N(n)$ balls, each centred at $\xi_i$, $i=1,\dots,N(n)$, that are the result of intersecting $l$ nested balls at each $\xi_i$, $i=1,\dots,N(n)$. In turn, $\displaystyle\bigcap_{j=1}^{q_n}\bigcup_{k=j}^{l} f^{-k}(\tilde{U}_n^c(\tau_{k-j}))$ determines that, for each $j=1,\dots,q_n$, one must take the complementary of at least one (so exactly one) nested ball around some $\xi_i$, $i=1,\dots,N(n)$. Thus, $\tilde{A}_{n,l}^{(q_n)}$ consists of a union of at most $N(n)-1$ balls with an annulus each centred at some $\xi_i$, $i=1,\dots,N(n)$. In particular, it can be the case that $\tilde{A}_{n,l}^{(q_n)}$ is made of $N(n)$ annuli centred at each $\xi_i$, $i=1,\dots,N(n)$, which gives the biggest possible number of connected components for $\tilde{A}_{n,l}^{(q_n)}$, that is $2N(n)$ connected components. Setting $t_n=N(n)$, we conclude
\begin{equation*}
\displaystyle\lim_{n \to \infty} \norm{\mathbf{1}_{\tilde{A}_{n,l}^{(q_n)}}}_{BV}n\left(\dfrac{1}{3}\right)^{N(n)} \leq \displaystyle\lim_{n \to \infty}(4N(n)+1)n\left(\dfrac{1}{3}\right)^{N(n)}<\displaystyle\lim_{n \to \infty}2^{N(n)}n\left(\dfrac{1}{3}\right)^{N(n)}=0.
\end{equation*}

We are left to justify our claim. Without loss of generality, let $l=1$. First, we consider the case where $A_1=[0,\tau_0) \times [\tau_1,+\infty)$, which means that $\displaystyle\bigcap_{j=0}^{1} f^{-j}(\tilde{U}_n(\tau_j))$ already consists of $N(n)$ annuli. Thus, the number of connected components of $\tilde{A}_{n,l}^{(q_n)}$ can't be bigger than if $\mathcal{I}_0=(0,\tau_0)$ and $\mathcal{I}_1=(0,\tau_1)$ were to be considered. The same reasoning allows us to discard the case where $A_l=[\tau_0,+\infty) \times [\tau_1,+\infty)$, where, in fact, the number of connected components of $\tilde{A}_{n,l}^{(q_n)}$ is necessarily smaller than with $A_l=[0,\tau_0) \times [0,\tau_1)$. Finally, $A_l=[\tau_0,+\infty) \times [0,\tau_1)$ doesn't make sense.
\end{proof}

For higher dimensional expanding systems, Saussol's space will play a similar role to $BV$.
\begin{definition}
Let $\phi:\mathcal{X} \to \mathbb{R}$ be an integrable function where $\mathcal{X}$ is a compact subset of $\mathbb{R}^n$ (\textit{i.e.} $\phi \in L^1(\text{Leb})$). Given a Borel set $\Gamma \subset \mathcal{X}$, let the \textit{oscillation} of $\phi$ over $\Gamma$ be defined as
\begin{equation*}
osc(\phi,\Gamma):=\text{ess}\sup_{\Gamma} \phi-\text{ess}\inf_{\Gamma} \phi.
\end{equation*}
Given real numbers $0<\alpha \leq 1$ and $\varepsilon_0>0$, define the $\alpha$-seminorm of $\phi$ as
\begin{equation*}
\modl{\phi}_{\alpha}:=\sup_{0<\varepsilon \leq \varepsilon_0}\varepsilon^{-\alpha} \int_{\mathbb{R}^n} osc(\phi,B_{\varepsilon}(x)) d\text{Leb}(x).
\end{equation*}
The space of functions with \textit{bounded $\alpha$-seminorm} $V_\alpha:=\{\phi \in L^1(\text{Leb}):\modl{\phi}_{\alpha}<\infty\}$ equipped with the norm $\norm{\phi}_{\alpha}:=\norm{\phi}_1+\modl{\phi}_{\alpha}$ is a Banach space.
\end{definition}
\begin{remark}
The definition of the space $V_{\alpha}$ (and corresponding norm) is independent of the choice of $\varepsilon_0$.
\end{remark}
The system in Examples \ref{ex:2-exp-nper} and \ref{ex:2-exp} has exponential decay of correlations for observables in $V_{\alpha}$, for some $\alpha \in (0,1]$, against $L^1$. We conclude with checking that (1) and (2) of Lemma \ref{lem:conditions} hold for the same examples.
\begin{lemma}\label{lem:shortcut-corr-2d}
Let $f$, $\mu$, $\mathcal{M}$ and $\Psi$ be as in Example \ref{ex:2-exp}. Then
\begin{enumerate}[(1)]
\item $\displaystyle\lim_{n \to \infty} \norm{\mathbf{1}_{A_{n,l}^{(q_n)}}}_{\mathcal{C}_1}n\rho(t_n)=0$ for some sequence $(t_n)_n$ with $t_n=o(n)$;
\item $\displaystyle\lim_{n \to \infty} \norm{\mathbf{1}_{U_n(\tau)}}_{\mathcal{C}_1}\sum_{j=q_n}^{n}\rho(j)=0$.
\end{enumerate}
are satisfied.
\end{lemma}
\begin{proof}
Here $\mathcal{C}_1=V_{\alpha}$, where $\alpha \in (0,1]$. Therefore,
\begin{equation*}
\norm{\mathbf{1}_{A_{n,l}^{(q_n)}}}_{V_{\alpha}}=\norm{\mathbf{1}_{A_{n,l}^{(q_n)}}}_1 + \modl{\mathbf{1}_{A_{n,l}^{(q_n)}}}_{\alpha}=\mu(A_{n,l}^{(q_n)}) + \modl{\mathbf{1}_{A_{n,l}^{(q_n)}}}_{\alpha}
\end{equation*}
so we want to bound the term $\modl{\mathbf{1}_{A_{n,l}^{(q_n)}}}_{\alpha}$. In words, $\modl{\mathbf{1}_{A_{n,l}^{(q_n)}}}_{\alpha}$ looks into the supremum, where $\varepsilon \in (0,\varepsilon_0]$, of the areas of $\varepsilon$-neighbourhoods of the boundary of the set $A_{n,l}^{(q_n)}$ multiplied by the factor $\varepsilon^{-\alpha}$. We observe that, in general, $A_{n,l}^{(q_n)}$ is made up of annuli or balls, more precisely of an amount of $\modl{\mathcal{M}}$ of annuli or balls centred at each one of the elements of $\mathcal{M}$. Since the outer circumference of $A_{n,l}^{(q_n)}$ has radius equal to $u_n(\tau_0)^{-2}$ (because $A_{n,l}^{(q_n)} \subset U_n(\tau_0)$, by definition) then, as long as $\varepsilon<u_n(\tau_0)^{-2}$ for all $\varepsilon \in (0,\varepsilon_0]$, any $\varepsilon$-neighbourhood of the boundary of $A_{n,l}^{(q_n)}$ can't have area bigger than $2\modl{\mathcal{M}}\varepsilon^{-\alpha}\varepsilon.2\pi u_n(\tau_0)^{-2}$ (when $A_{n,l}^{(q_n)}$ consists of annulus around each point in $\mathcal{M}$). Since $2\modl{\mathcal{M}}\varepsilon^{-\alpha}\varepsilon.2\pi u_n(\tau_0)^{-2}$ attains the supremum at $\varepsilon_0$, we have
\begin{equation*}
\modl{\mathbf{1}_{A_{n,l}^{(q_n)}}}_{\alpha} \leq 2\modl{\mathcal{M}}\varepsilon_0^{-\alpha+1}.2\pi u_n(\tau_0)^{-2}.
\end{equation*}
However, our assumption that $\varepsilon<u_n(\tau_0)^{-2}$ for all $\varepsilon \in (0,\varepsilon_0]$ does not hold if $n$ is large. As a consequence, any $\varepsilon$-neighbourhood of the boundary of $A_{n,l}^{(q_n)}$, for large enough $n$, gets to be a union of $\modl{\mathcal{M}}$ balls of radius $\varepsilon$. Thus,
\begin{equation*}
\modl{\mathbf{1}_{A_{n,l}^{(q_n)}}}_{\alpha} \leq \modl{\mathcal{M}}\varepsilon_0^{-\alpha}.\pi\varepsilon_0^2=\modl{\mathcal{M}}\pi\varepsilon_0^{2-\alpha}.
\end{equation*}
We conclude that
\begin{equation*}
\modl{\mathbf{1}_{A_{n,l}^{(q_n)}}}_{\alpha} \leq \modl{\mathcal{M}}\pi\varepsilon_0^{2-\alpha}.
\end{equation*}

As mentioned above, $A_{n,l}^{(q_n)} \subset U_n(\tau_0)$ so that $\mu(A_{n,l}^{(q_n)}) \leq \mu(U_n(\tau_0))=\pi (u_n(\tau_0)^{-2})^2$.

Finally,
\begin{equation*}
\begin{split}
\norm{\mathbf{1}_{A_{n,l}^{(q_n)}}}_{V_{\alpha}} &\leq \pi(u_n(\tau_0)^{-2})^2 +2\pi\varepsilon_0^{2-\alpha}\\
&\leq \pi(u_1(\tau_0)^{-2})^2 +2\pi\varepsilon_0^{2-\alpha}.
\end{split}
\end{equation*}
Since $\rho$ decays exponentially we can choose, for example, $t_n=\log(n)$ and (1) follows.

We obtain the exact same estimates for $\norm{\mathbf{1}_{U_n(\tau)}}_{V_{\alpha}}$ so that (2) is now trivially satisfied.
\end{proof}

\bibliography{ref_11-21}

\newcommand{\etalchar}[1]{$^{#1}$}
\begin{thebibliography}{CFF{\etalchar{+}}15}

\bibitem[AFFR16]{16correlated}
Davide Azevedo, Ana Cristina~Moreira Freitas, Jorge~Milhazes Freitas, and
  Fagner~B. Rodrigues.
\newblock Clustering of extreme events created by multiple correlated maxima.
\newblock {\em Phys. D}, 315:33--48, 2016.

\bibitem[AFFR17]{17correlated}
Davide Azevedo, Ana Cristina~Moreira Freitas, Jorge~Milhazes Freitas, and
  Fagner~B Rodrigues.
\newblock Extreme value laws for dynamical systems with countable extremal
  sets.
\newblock {\em Journal of Statistical Physics}, 167(5):1244--1261, 2017.

\bibitem[BPS18]{basrak18}
Bojan Basrak, Hrvoje Planini{\'c}, and Philippe Soulier.
\newblock An invariance principle for sums and record times of regularly
  varying stationary sequences.
\newblock {\em Probability Theory and Related Fields}, 172(3):869--914, 2018.

\bibitem[CFF{\etalchar{+}}15]{torus15}
Maria Carvalho, Ana Cristina~Moreira Freitas, Jorge~Milhazes Freitas, Mark
  Holland, and Matthew Nicol.
\newblock Extremal dichotomy for uniformly hyperbolic systems.
\newblock {\em Dynamical Systems}, 30(4):383--403, 2015.

\bibitem[CHN21]{hyp21}
Meagan Carney, Mark Holland, and Matthew Nicol.
\newblock Extremes and extremal indices for level set observables on hyperbolic
  systems.
\newblock {\em Nonlinearity}, 34(2):1136, 2021.

\bibitem[FFM20]{20conv}
Ana~Cristina Freitas, Jorge Freitas, and M{\'a}rio Magalh{\~a}es.
\newblock Complete convergence and records for dynamically generated stochastic
  processes.
\newblock {\em Transactions of the American Mathematical Society},
  373(1):435--478, 2020.

\bibitem[FFT10]{10hit}
Ana Cristina~Moreira Freitas, Jorge~Milhazes Freitas, and Mike Todd.
\newblock Hitting time statistics and extreme value theory.
\newblock {\em Probability Theory and Related Fields}, 147(3-4):675--710, 2010.

\bibitem[FFT12]{12index}
Ana Cristina~Moreira Freitas, Jorge~Milhazes Freitas, and Mike Todd.
\newblock The extremal index, hitting time statistics and periodicity.
\newblock {\em Advances in Mathematics}, 231(5):2626--2665, 2012.

\bibitem[FFT13]{13compound}
Ana Cristina~Moreira Freitas, Jorge~Milhazes Freitas, and Mike Todd.
\newblock The compound poisson limit ruling periodic extreme behaviour of
  non-uniformly hyperbolic dynamics.
\newblock {\em Communications in Mathematical Physics}, 321(2):483--527, 2013.

\bibitem[FFT20]{20enriched}
Ana Cristina~Moreira Freitas, Jorge~Milhazes Freitas, and Mike Todd.
\newblock Enriched functional limit theorems for dynamical systems.
\newblock {\em arXiv preprint arXiv:2011.10153}, 2020.

\bibitem[HV09]{hv09}
Nicolai Haydn and Sandro Vaienti.
\newblock The compound poisson distribution and return times in dynamical
  systems.
\newblock {\em Probability theory and related fields}, 144(3-4):517--542, 2009.

\bibitem[TK10]{tk10}
Marta Tyran-Kami{\'n}ska.
\newblock Weak convergence to l{\'e}vy stable processes in dynamical systems.
\newblock {\em Stochastics and Dynamics}, 10(02):263--289, 2010.

\bibitem[Whi02]{whitt}
Ward Whitt.
\newblock {\em Stochastic-process limits: an introduction to stochastic-process
  limits and their application to queues}.
\newblock Springer Science \& Business Media, 2002.

\end{thebibliography}
\bibliographystyle{alpha}

\end{document}